\documentclass[a4paper,12pt]{article}

\usepackage{amsfonts,amssymb,amsmath,amsthm,latexsym,epsfig,amscd,bbm,stmaryrd}

\usepackage[english]{babel}
\makeatletter \@addtoreset{equation}{section}
\makeatother

\makeatletter \@addtoreset{enunciato}{section}
\makeatother

\newcounter{enunciato}[section]

\newtheorem{ittheorem}{Theorem}
\newtheorem{itlemma}{Lemma}
\newtheorem{itproposition}{Proposition}
\newtheorem{itdefinition}{Definition}
\newtheorem{itremark}{Remark}
\newtheorem{itclaim}{Claim}
\newtheorem{itfact}{Fact}
\newtheorem{itconjecture}{Conjecture}
\newtheorem{itcorollary}{Corollary}

\newenvironment{theorem}{\addtocounter{enunciato}{1}
\begin{ittheorem}}{\end{ittheorem}}

\newenvironment{lemma}{\addtocounter{enunciato}{1}
\begin{itlemma}}{\end{itlemma}}

\newenvironment{proposition}{\addtocounter{enunciato}{1}
\begin{itproposition}}{\end{itproposition}}

\newenvironment{definition}{\addtocounter{enunciato}{1}
\begin{itdefinition}}{\end{itdefinition}}

\newenvironment{remark}{\addtocounter{enunciato}{1}
\begin{itremark}}{\end{itremark}}

\newenvironment{conjecture}{\addtocounter{enunciato}{1}
\begin{itconjecture}}{\end{itconjecture}}

\newenvironment{corollary}{\addtocounter{enunciato}{1}
\begin{itcorollary}}{\end{itcorollary}}

\newcommand{\be}[1]{\begin{equation}\label{#1}}
\newcommand{\ee}{\end{equation}}

\newcommand{\bl}[1]{\begin{lemma}\label{#1}}
\newcommand{\el}{\end{lemma}}

\newcommand{\br}[1]{\begin{remark}\label{#1}}
\newcommand{\er}{\end{remark}}

\newcommand{\bt}[1]{\begin{theorem}\label{#1}}
\newcommand{\et}{\end{theorem}}

\newcommand{\bd}[1]{\begin{definition}\label{#1}}
\newcommand{\ed}{\end{definition}}

\newcommand{\bp}[1]{\begin{proposition}\label{#1}}
\newcommand{\ep}{\end{proposition}}

\newcommand{\bc}[1]{\begin{corollary}\label{#1}}
\newcommand{\ec}{\end{corollary}}

\newcommand{\bcj}[1]{\begin{conjecture}\label{#1}}
\newcommand{\ecj}{\end{conjecture}}

\newcommand{\bpr}{\begin{proof}}
\newcommand{\epr}{\end{proof}}

\def \Z {{\mathbb Z}}
\def \R {{\mathbb R}}

\def \N {{\mathbb N}}

\def \ba {\begin{array}}
\def \ea {\end{array}}
\def \da {\downarrow}

\def \P  {{\mathbb P}}
\def \E  {{\mathbb E}}

\begin{document}


\title{Scaling of a random walk on a\\ 
supercritical contact process}

\author{\renewcommand{\thefootnote}{\arabic{footnote}}
F.\ den Hollander \footnotemark[1]
\\
\renewcommand{\thefootnote}{\arabic{footnote}}
R.\ dos Santos \footnotemark[1]}

\footnotetext[1]{
Mathematical Institute, Leiden University, P.O.\ Box 9512,
2300 RA Leiden, The Netherlands}

\maketitle

\begin{abstract}
A proof is provided of a strong law of large numbers for a one-dimensional random 
walk in a dynamic random environment given by a supercritical contact process in 
equilibrium. The proof is based on a coupling argument that traces the space-time 
cones containing the infection clusters generated by single infections and uses 
that the random walk eventually gets trapped inside the union of these cones. For 
the case where the local drifts of the random walk are smaller than the speed at 
which infection clusters grow, the random walk eventually gets trapped inside a 
single cone. This in turn leads to the existence of regeneration times at which 
the random walk forgets its past. The latter are used to prove a functional 
central limit theorem and a large deviation principle.

The qualitative dependence of the speed and the volatility on the 
infection parameter is investigated, and some open problems are mentioned.

\vspace{0.5cm}
\noindent
{\it Acknowledgment.} The authors would like to thank Stein Bethuelsen and Markus Heydenreich 
for stimulating discussions, and the anonymous referee for useful suggestions.

\vspace{0.5cm}\noindent
\end{abstract}


\section{Introduction}
\label{sec:intro}


\subsection{Background, motivation and outline}
\label{sec:back}

\paragraph{Background.}
\emph{A random walk in a dynamic random environment} on $\Z^d$, $d \geq 1$, is a 
random process where a ``particle'' makes random jumps with transition rates that 
depend on its location and themselves evolve with time. A typical example is when 
the dynamic random environment is given by an interacting particle system
\begin{equation}
\xi = (\xi_t)_{t \geq 0} \text{ with } \xi_t = \{\xi_t(x)\colon\,x\in\Z^d\}
\in \Omega,
\end{equation}   
where $\Omega$ is the configuration space, and $\xi_0$ is typically drawn from 
equilibrium. In the case where $\Omega=\{0,1\}^{\Z^d}$, the configurations can be 
thought of as consisting of ``particles'' and ``holes''. Given $\xi$, run a random 
walk $W=(W_t)_{t\geq 0}$ on $\Z^d$ that jumps at a fixed rate, but uses different 
transition kernels on a particle and on a hole. The key question is: What are 
the scaling properties of $W$ and how do these properties depend on the law 
of $\xi$?

The literature on random walks in dynamic random environments is still modest
(for a recent overview, see Avena~\cite{Avthesis}, Chapter 1). In Avena, den 
Hollander and Redig~\cite{AvdHoRe11} a strong law of large numbers (SLLN) was 
proved for a class of interacting particle systems satisfying a mild space-time 
mixing condition, called \emph{cone-mixing}. Roughly speaking, this is the 
requirement that for every $m>0$ all states inside the space-time cone (see
Fig.~\ref{fig-cone})
\begin{equation}
\label{spacetimecone}
\mathrm{CONE}_t := \big\{(x,s) \in \Z^d \times 
[t,\infty) \colon\,\|x\| \leq m(s-t)\big\},
\end{equation} 
are conditionally independent of the states at time zero in the limit as $t\to
\infty$. The proof of the SLLN uses a \emph{regeneration-time} argument. Under 
a cone-mixing condition involving multiple cones, a functional central limit 
theorem (FCLT) can be derived as well, and under monotonicity conditions also 
a large deviation principle (LDP). 

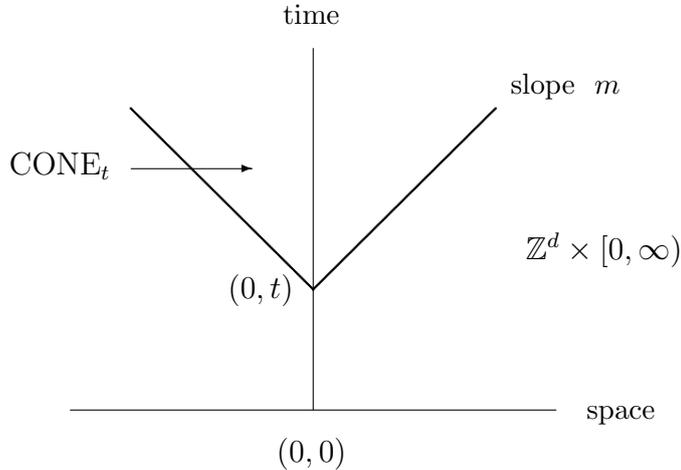
\begin{figure}[hbtp]
\vspace{1.8cm}
\begin{center}
\setlength{\unitlength}{0.4cm}
\begin{picture}(12,10)(-5,0)
\put(-8,0){\line(16,0){16}}
\put(0,0){\line(0,12){12}}
{\thicklines
\qbezier(0,4)(3,7)(6,10)
\qbezier(0,4)(-3,7)(-6,10)
}
\put(-1.2,-1.7){$(0,0)$}
\put(-2.8,3.7){$(0,t)$}
\put(-6,8){\vector(1,0){4}}
\put(7,5){$\Z^d \times [0,\infty)$}
\put(-10,7.8){$\mathrm{CONE}_t$}
\put(-1,12.8){\small\mbox{time}}
\put(9,-.3){\small\mbox{space}}
\put(6.5,10.5){\small\mbox{slope } $m$}
\end{picture}
\end{center}
\caption{\small The cone defined in \eqref{spacetimecone}.}
\label{fig-cone}
\vspace{0.3cm}
\end{figure}


Many interacting particle systems are cone-mixing, including spin-flip systems 
with spin-flip rates that are weakly dependent on the configuration, e.g.\ the 
stochastic Ising model above the critical temperature. However, also many 
interacting particle systems are not cone-mixing, including independent simple 
random walks, the exclusion process, the contact process and the voter model. 
Indeed, these systems have \emph{slowly decaying space-time correlations}. For 
instance, in the exclusion process particles are conserved and cannot sit on top 
of each other. Therefore, if at time zero there are particles everywhere in the box 
$[-t^2,t^2] \cap \Z^d$, then these particles form a ``large traffic jam around 
the origin''. This traffic jam will survive up to time $t$ with a probability 
tending to $1$ as $t\to\infty$, and will therefore affect the states near the tip 
of $\mathrm{CONE}_t$. Similarly, in the contact process, if at time zero there are 
no infections in the box $[-t^2,t^2] \cap \Z^d$, then no infections will be seen 
near the tip of $\mathrm{CONE}_t$ as well.

\paragraph{Motivation.}
Several attempts have been made to extend the SLLN to interacting particle systems 
that are not cone-mixing, with partial success. Examples include: independent 
simple random walks (den Hollander, Kesten and Sidoravicius~\cite{dHoKeSipr}) 
and the exclusion process (Avena, dos Santos and V\"ollering~\cite{AvdSaVo12}, 
Avena~\cite{Av12}). The present paper considers the \emph{supercritical contact 
process}. We exploit the graphical representation, which allows us to simultaneously 
couple all realizations of the contact process starting from different initial 
configurations. This coupling in turn allows us to first prove the SLLN when the 
initial configuration is ``all infected'' (with the help of a subadditivity 
argument), and then show that the same result holds when the initial configuration 
is drawn from equilibrium. The main idea is to use the coupling to show that 
configurations agree in large space-time cones containing the infection clusters 
generated by single infections and that the random walk eventually gets trapped 
inside the \emph{union} of these cones.

Under the assumption that the local drifts of the random walk are smaller than the 
speed at which infection clusters grow, the random walk eventually gets trapped 
inside a \emph{single} cone. We show that this implies the existence of 
\emph{regeneration times} at which the random walk ``forgets its past''. The 
latter in turn allow us to prove the FCLT and the LDP. 

It is typically difficult to obtain information about the speed in the SLLN, the 
volatility in the FCLT and the rate function in the LDP. In general, these are 
non-trivial functions of the parameters in the model, a situation that is well 
known from the literature on random walks in static random environments (for 
overviews, see Sznitman~\cite{Sz02} and Zeitouni~\cite{Ze04}). The reason is that 
these quantities depend on the \emph{environment process} (i.e., the process of 
environments as seen from the location of the walk), which is typically hard to 
analyze. For the supercritical contact process we are able to derive a few 
qualitative properties as a function of the infection parameter, but it remains 
a challenge to obtain a full quantitative description.        

A model of a random walk on the infinite cluster of supercritical oriented 
percolation (the discrete-time analogue of the contact process) is treated in 
Birkner, \v{C}ern\'y, Depperschmidt and Gantert \cite{BiCeDeGa12}, where a SLLN and 
a quenched and annealed CLT are obtained. This model can be viewed as a random 
walk in a dynamic random environment, but it has non-elliptic transition 
probabilities different from the ones we consider here, because the random walk
is confined to the infinite cluster.

\paragraph{Outline.}
In Section~\ref{sec:model} we define the model. In Section~\ref{sec:thms} 
we state our main results: two theorems claiming the SLLN, the FCLT and 
the LDP under appropriate conditions on the model parameters. In 
Section~\ref{sec:disc} we mention some open problems. The proofs of the 
theorems are given in Sections~\ref{sec:LLN}, \ref{sec:speed}, and \ref{sec:regeneration}. 
Sections~\ref{sec:construc} 
and \ref{sec:moreCP} contain preparatory work.


\subsection{Model}
\label{sec:model}

In this paper we consider the case where the dynamic random environment is the 
one-dimensional linear \emph{contact process} $\xi = (\xi_t)_{t \geq 0}$, i.e.,
the spin-flip system on $\Omega:=\{0,1\}^{\Z}$ with local transition rates given 
by
\begin{equation}
\label{ratescp}
\eta \to \eta^{x} \text{ with rate } 
\left\{\begin{array}{ll}
1 & \text{ if } \eta(x) = 1,\\
\lambda \left\{\eta(x-1)+\eta(x+1)\right\} & \text{ if } \eta(x) = 0,\\
\end{array}
\right.
\end{equation}
where $\lambda \in (0,\infty)$ and $\eta^x$ is defined by $\eta^x(y) := \eta(y)$ 
for $y \neq x$, $\eta^x(x):=1-\eta(x)$. We call a site \emph{infected} when its 
state is $1$, and \emph{healthy} when its state is 0. See Liggett~\cite{Li85}, 
Chapter VI, for proper definitions.

The empty configuration $\mathbf{0} \in \Omega$, given by $\mathbf{0}(x) = 0$ for 
all $x \in \Z$, is an absorbing state for $\xi$, while the full configuration 
$\mathbf{1} \in \Omega$, given by $\mathbf{1}(x) = 1$ for all $x \in \Z$, evolves 
towards an equilibrium measure $\nu_\lambda$, called the ``upper invariant measure'', 
that is stationary and ergodic under space-shifts. 
There is a critical threshold $\lambda_c \in (0,\infty)$ such that: 
(1) for $\lambda \in (0,\lambda_c]$, $\nu_\lambda = \delta_{\mathbf{0}}$; (2) for 
$\lambda \in (\lambda_c,\infty)$, $\rho_\lambda := \nu_\lambda(\eta(0)=1)>0$. In 
the latter case, $\delta_{\mathbf{0}}$ and $\nu_\lambda$ are the only equilibrium 
measures. It is known that $\nu_\lambda$ has exponentially decaying correlations, 
and that $\lambda \mapsto \rho_\lambda$ is continuous and non-decreasing with 
$\lim_{\lambda \to \infty} \rho_{\lambda} = 1$.

For a fixed realization of $\xi$, we define the random walk $W:= (W_t)_{t \geq 0}$ 
as the time-inhomogeneous Markov process on $\Z$ that, given $W_t=x$, jumps to
\begin{equation}
\begin{array}{lll}
x + 1 & \text{ at rate  } \;\; \alpha_1 \xi_t(x) + \alpha_0 \left[1-\xi_t(x)\right],\\
x - 1 & \text{ at rate  } \;\; \beta_1 \xi_t(x) + \beta_0 \left[1-\xi_t(x)\right],
\end{array}
\end{equation}
where $\alpha_i, \beta_i \in (0,\infty)$, $i=0,1$. We assume that
\begin{equation}
\label{condjumprates1}
\alpha_0+\beta_0 = \alpha_1+\beta_1 =: \gamma ,
\end{equation}
and that
\begin{equation}
\label{condjumprates2}
v_1 > v_0 \text{ with } v_1 := \alpha_1 - \beta_1 \text{ and } 
v_0 := \alpha_0 - \beta_0,
\end{equation}
i.e., the jump rate is constant and equal to $\gamma$ everywhere, while the drift 
to the right is larger on infected sites than on healthy sites. Observe that the
assumption in \eqref{condjumprates2} is made without loss of generality: since the 
contact process is invariant under reflection in the origin, $-W$ has the same law 
as $W$ with inverted jump rates.


\subsection{Theorems}
\label{sec:thms}

For a probability measure $\mu$ on $\Omega$, 
let $\P_{\mu}$ denote the joint law of $W$ and $\xi$ when the latter is 
started from $\mu$. Our SLLN reads as follows.
\begin{theorem}
\label{LLN} 
Suppose that {\rm (\ref{condjumprates1}--\ref{condjumprates2})} hold.\\ 
(a) For every $\lambda \in (\lambda_c,\infty)$ there exists a $v(\lambda) \in 
[v_0,v_1]$ such that, for any probability measure $\mu$ on $\Omega$
that is stochastically larger than a non-trivial shift-invariant and ergodic measure,
\begin{equation}
\label{eq:LLN}
\lim_{t \to \infty} t^{-1}W_t = v(\lambda) 
\quad \P_{\mu}\text{-a.s.\ and in } L^p,\,p \geq 1.
\end{equation}
In particular, \eqref{eq:LLN} holds for $\mu = \nu_\lambda$.\\
(b) The function $\lambda \mapsto v(\lambda)$ is non-decreasing and right-continuous 
on $(\lambda_c,\infty)$, with $v(\lambda) \in (v_0,v_1)$ for all $\lambda \in 
(\lambda_c,\infty)$ and $\lim_{\lambda\to\infty} v(\lambda) = v_1$.
\end{theorem}

\noindent
We note in passing that if $\lambda \in (0,\lambda_c)$, then $\xi_t$ agrees with 
$\mathbf{0}$ on an interval that grows exponentially fast in $t$ (Liggett~\cite{Li85}, 
Chapter VI), and so it is trivial to deduce that $W$ satisfies the SLLN with $v(\lambda)
=v_0$. 

A FCLT and an LDP hold under an additional restriction, namely, $\lambda \in 
(\lambda_W,\infty)$
with
\begin{equation}
\label{extcond}
\lambda_W := \inf\big\{ \lambda \in (\lambda_c,\infty) \colon\, 
|v_0| \vee |v_1| < \iota(\lambda)\big\}.
\end{equation}
Here, $\lambda\mapsto\iota(\lambda)$ is the infection propagation speed (see \eqref{convR0}
in Section~\ref{subsec:cp}), which is known to be continuous, strictly positive and 
strictly increasing on $(\lambda_c,\infty)$, with $\lim_{\lambda\downarrow\lambda_c} 
\iota(\lambda)=0$ and $\lim_{\lambda\to\infty} \iota(\lambda)=\infty$.

\begin{theorem}
\label{LLNregen}
Suppose that {\rm (\ref{condjumprates1}--\ref{condjumprates2})} hold.\\
(a) For every $\lambda \in (\lambda_W,\infty)$ there exists a $\sigma(\lambda) 
\in (0,\infty)$ such that, under $\P_{\nu_\lambda}$,
\begin{equation}
\left(\frac{W_{nt}-v(\lambda)nt}{\sigma(\lambda)\sqrt{n}}\right)_{t \geq 0}
\Longrightarrow (B_t)_{t \geq 0} \qquad \mbox{ as } n\to\infty,
\end{equation}
where $B$ is standard Brownian motion and $\Longrightarrow$ denotes weak 
convergence in path space.\\
(b) The functions $\lambda\mapsto v(\lambda)$ and $\lambda\mapsto\sigma(\lambda)$ 
are continuous on $(\lambda_W,\infty)$.\\
(c) For every $\lambda\in(\lambda_W,\infty)$, $(t^{-1}W_t)_{t > 0}$ under
$\mathbb{P}_{\nu_\lambda}$ satisfies the large deviation principle on $\R$
with a finite and convex rate function that has a unique zero at $v(\lambda)$.
\end{theorem}

The intuitive reason why the rate function has a unique zero is that deviations 
of the empirical speed in the: (i) upward direction require a density of infected 
sites larger than $\rho_\lambda$, which is costly because infections become 
healthy independently of the states at the other sites; (ii) downward direction 
require a density of infected sites smaller than $\rho_\lambda$, which is costly 
because infection clusters grow at a linear speed and rapidly fill up healthy 
intervals everywhere.


\subsection{Discussion}
\label{sec:disc}

{\bf 1.} 
It is natural to expect that $\lambda \mapsto v(\lambda)$ is continuous and 
strictly increasing on $(\lambda_c,\infty)$ with $\lim_{\lambda\da\lambda_c} 
v(\lambda)=v_0$. Fig.~\ref{fig-speed} shows a qualitative plot of the speed
in that setting. If $0 \in (v_0,v_1)$, then there is a critical threshold 
$\lambda^\ast \in (\lambda_c,\infty)$ at which the speed changes sign. It is 
natural to ask whether $\lambda \mapsto v(\lambda)$ is concave on $(\lambda_c,
\infty)$ and Lipshitz at $\lambda_c$.

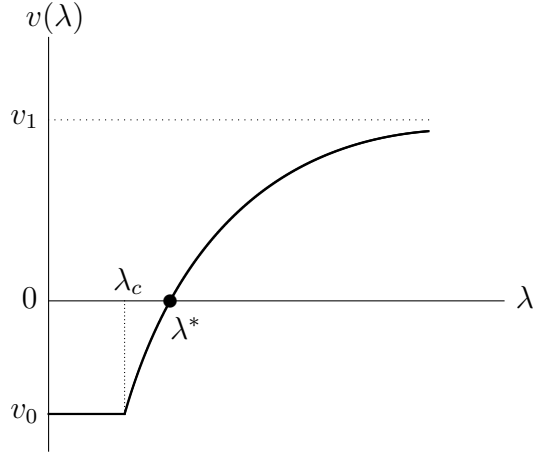
\begin{figure}[hbtp]
\vspace{0.5cm}
\begin{center}
\setlength{\unitlength}{0.5cm}
\begin{picture}(12,11)(0,-3.7)
\put(0,0){\line(12,0){12}} 
\put(0,-4){\line(0,11){11}}
{\thicklines
\qbezier(2,-3)(4,4)(10,4.5)
\qbezier(0,-3)(1,-3)(2,-3)
}
\qbezier[60](0,4.8)(5,4.8)(10,4.8)
\qbezier[30](2,0)(2,-1.5)(2,-3)
\put(-.7,-.2){$0$}
\put(-.6,7.3){$v(\lambda)$}
\put(12.3,-.2){$\lambda$}
\put(1.7,.3){$\lambda_c$}
\put(3.2,-1){$\lambda^\ast$}
\put(-1,4.7){$v_1$}
\put(-1,-3.1){$v_0$}
\put(3.2,0){\circle*{.35}}
\end{picture}
\end{center}
\caption{\small Qualitative plot of $\lambda\mapsto v(\lambda)$ when $0 \in (v_0,v_1)$.}
\label{fig-speed}
\end{figure}

\medskip\noindent
{\bf 2.}
We know that $W$ is transient when $v(\lambda) \neq 0$. Is $W$ recurrent when 
$v(\lambda)=0$?

\medskip\noindent
{\bf 3.}
We expect \eqref{extcond} to be redundant. Moreover, we expect that for every
$\lambda \in (\lambda_c,\infty)$ the \emph{environment process} (i.e., the
process of environments as seen from the location of the random walk) has a 
unique and non-trivial equilibrium measure that is absolutely continuous with 
respect to $\nu_\lambda$.

\medskip\noindent
{\bf 4.}
Theorem~\ref{LLN} can be extended to arbitrary initial configurations that have a 
``strictly positive lower density''  (see Remark~\ref{rem:initconfigposdensity}
in Section~\ref{subsec:LLNeq} below). Also, Theorem~\ref{LLN} remains valid
for $\mu$ stochastically larger than $\nu_\lambda$ even when some of the jump rates 
$\alpha_i,\beta_i$, $i \in \{0,1\}$, are equal to zero (see Remark~\ref{rem:nonellipticrates} 
in Section~\ref{subsec:LLNeq} below).

\medskip\noindent
{\bf 5.}
Theorem~\ref{LLNregen}(a) can be extended (with the same mean and variance) 
to arbitrary initial configurations containing infinitely many infections, while 
Theorem~\ref{LLNregen}(c) can be extended (with a different rate function) to any 
initial measure that has positive correlations and is stochastically larger than a 
non-trivial Bernoulli product measure (see Remark~\ref{rem:extensionsCLTLDP} 
in Section~\ref{subsec:regenconst} below).

\medskip\noindent
{\bf 6.}
Theorems \ref{LLN}--\ref{LLNregen} can presumably be extended to $\Z^d$ with
$d \geq 2$. Also in higher dimensions single infections create infection 
clusters that grow at a linear speed (i.e., asymptotically form a ball with 
a linearly growing radius). The construction of the regeneration times when
$\lambda \in (\lambda_W,\infty)$, with $\lambda_W$ the analogue of \eqref{extcond}, 
appears to be possible. 

\medskip\noindent
{\bf 7.}
It would be interesting to extend Theorems~\ref{LLN}--\ref{LLNregen} to 
\emph{multi-type contact processes}. On each type $i$ the random walk
has transition rates $\alpha_i,\beta_i$ such that $\alpha_i+\beta_i=\gamma$ 
for all $i$. As long as the dynamics is monotone and $i \mapsto v_i$ is
non-decreasing, many of the arguments in the present paper carry over.


\section{Construction}
\label{sec:construc}

In Section~\ref{subsec:cp} we construct the contact process, in 
Section~\ref{subsec:RW} the random walk on top of the contact process.


\subsection{Contact process}
\label{subsec:cp}

A c\`adl\`ag version of the contact process can be constructed from a graphical 
representation in the following standard fashion. Let $:=(H(x))_{x \in \Z}$ and 
$I := (I(x))_{x \in \Z}$ be two independent collections of i.i.d.\ Poisson processes 
with rates $1$ and $\lambda$, respectively. On $\Z \times [0,\infty)$, draw the 
events of $H(x)$ as crosses over $x$ and the events of $I(x)$ as two-sided arrows 
between $x$ and $x+1$ (see Fig.~\ref{fig-graphrep}).

\begin{figure}[hbtp]
\vspace{1cm}
\begin{center}
\setlength{\unitlength}{0.4cm}
\begin{picture}(20,10)(0,0)
\put(0,0){\line(22,0){22}} 
\put(0,11){\line(22,0){22}}
\put(2,0){\line(0,1){11}}
\put(5,0){\line(0,1){11}}
\put(8,0){\line(0,1){11}}
\put(11,0){\line(0,1){11}}
\put(14,0){\line(0,1){11}}
\put(17,0){\line(0,1){11}}
\put(20,0){\line(0,1){11}}
\put(2,3){\vector(1,0){3}}
\put(5,3){\vector(-1,0){3}}
\put(2,8.5){\vector(1,0){3}}
\put(5,8.5){\vector(-1,0){3}}
\put(5,7.5){\vector(1,0){3}}
\put(8,7.5){\vector(-1,0){3}}
\put(8,2){\vector(1,0){3}}
\put(11,2){\vector(-1,0){3}}
\put(11,5){\vector(1,0){3}}
\put(14,5){\vector(-1,0){3}}
\put(14,6){\vector(1,0){3}}
\put(17,6){\vector(-1,0){3}}
\put(14,9.7){\vector(1,0){3}}
\put(17,9.7){\vector(-1,0){3}}
\put(17,7){\vector(1,0){3}}
\put(20,7){\vector(-1,0){3}}
\put(1.6,1){$\times$}
\put(1.6,9){$\times$}
\put(4.6,4){$\times$}
\put(10.6,6){$\times$}
\put(13.6,8){$\times$}
\put(16.6,3){$\times$}
\put(19.6,10.2){$\times$}
{\linethickness{0.04cm}
\put(11,0){\line(0,1){6.16}}
\put(11,2){\line(-1,0){3}}
\put(8,2){\line(0,1){9}}
\put(8,7.5){\line(-1,0){3}}
\put(5,7.5){\line(0,1){3.5}}
\put(5,8.5){\line(-1,0){3}}
\put(2,8.5){\line(0,1){0.76}}
\put(11,5){\line(1,0){3}}
\put(14,5){\line(0,1){3.26}}
\put(14,6){\line(1,0){3}}
\put(17,6){\line(0,1){5}}
\put(17,9.7){\line(-1,0){3}}
\put(14,9.7){\line(0,1){1.3}}
\put(17,7){\line(1,0){3}}
\put(20,7){\line(0,1){3.46}}
}
\put(1,-1.2){$-3$} 
\put(4,-1.2){$-2$} 
\put(7,-1.2){$-1$}
\put(10.75,-1.2){$0$}
\put(13.75,-1.2){$1$}
\put(16.8,-1.2){$2$}
\put(19.8,-1.2){$3$}
\put(-1.2,-.3){$0$}
\put(-1.2,10.7){$t$}
\put(23,0){$\mathbb{Z}$}
\put(11,0){\circle*{.35}}
\put(5,11){\circle*{.35}}
\put(8,11){\circle*{.35}}
\put(14,11){\circle*{.35}}
\put(17,11){\circle*{.35}}
\end{picture}
\end{center}
\caption{\small Graphical representation. The crosses are events of $H$ and the 
arrows are events of $I$. The thick lines cover the region that is infected when 
the initial configuration has a single infection at the origin. } 
\label{fig-graphrep}
\end{figure}

\noindent
(The standard graphical representation uses Poisson processes of one-sided arrows 
to the right and to the left on every time line, each with rate $\lambda$. This 
gives the same dynamics.)

For $x,y \in \Z$ and $0 \le s \le t$, we say that $(x,s)$ and $(y,t)$ are 
\emph{connected}, written $(x,s) \leftrightarrow (y,t)$, if and only if there 
exists a nearest-neighbor path in $\Z \times [0,\infty)$ starting at $(x,s)$ 
and ending at $(y,t)$, going either upwards in time or sideways in space across 
arrows without hitting crosses. For $x \in \Z$, we define the cluster of $x$ 
at time $t$ by
\begin{equation}
\label{defclusters}
C_t(x) := \big\{ y \in \Z\colon\,(x,0) \leftrightarrow (y,t)\big\}.
\end{equation}
For example, in Fig.~\ref{fig-graphrep}, $C_t(0) = \{-2, -1, 1, 2\}$ and 
$C_t(2) = \emptyset$. Note that $C_t(x)$ is a function of $H$ and $I$.

For a fixed initial configuration $\eta$, we declare $\xi_t(y) = 1$ if there 
exists an $x$ such that $y \in C_t(x)$ and $\eta(x) = 1$, and we declare 
$\xi_t(y) = 0$ otherwise. Then $\xi$ is adapted to the filtration
\begin{equation}
\label{defcurlyF}
\mathcal{F}_t : = \sigma\big(\xi_0, (H_s,I_s)_{s \in [0,t]}\big).
\end{equation}
This construction allows us to \emph{simultaneously couple} copies of the contact 
process starting from \emph{all} configurations $\eta \in \Omega$. In the following 
we will write $\xi(\eta)$ and $\xi_t(\eta)(x)$ when we want to exhibit that the 
initial configuration is $\eta$. 

We note two consequences of the graphical construction, stated in 
Lemmas~\ref{monotCP}--\ref{defRLmost} below. The first is the monotonicity of 
$\eta\mapsto\xi(\eta)$, the second concerns the state of the sites surrounded 
by the cluster of an infected site. The notation $\eta \leq \eta^{\prime}$ stands 
for $\eta(x) \leq \eta^{\prime}(x)$ for all $x\in\Z$.

\begin{lemma}
\label{monotCP}
If $\eta \le \eta^{\prime}$, then $\xi_t(\eta) \le \xi_t(\eta^{\prime})$ for all
$t \geq 0$.
\end{lemma}

\bpr 
Immediate from the definition of $\xi_t$ in terms of $\eta$ and $(C_t(x))_{x \in \Z}$.
\epr

For $x \in \Z$, define the left-most and the right-most site influenced by site $x$ 
at time $t$ as
\begin{equation}
\label{defRLmost}
\begin{array}{ll}
L_t(x) := \inf C_t(x),\\
R_t(x) := \sup C_t(x),\\
\end{array}
\end{equation}
where $\inf \emptyset = \infty$ and $\sup \emptyset = - \infty$. By symmetry, for any 
$t \ge 0$, $R_t(x)-x$ and $x-L_t(x)$ have the same distribution, independently of $x$.

\begin{lemma}
\label{depinclusters}
Fix $x \in \Z$ and $t \ge 0$. If $C_t(x) \neq \emptyset$ and $y \in [L_t(x), R_t(x)] 
\cap \Z$, then $\eta\mapsto\xi_t(\eta)(y)$ is constant on $\{\eta\in\Omega\colon\,
\eta(x)=1\}$.
\end{lemma}

\bpr
It suffices to show that, under the conditions stated, $\xi_t(\eta)(y) = 1$ if and 
only if $y \in C_t(x)$. The `if' part is obvious. For the `only if' part, note that 
if there is a $z \neq x$ such that $(z,0) \leftrightarrow (y,t)$, then any path 
realizing the connection must cross a path connecting $(x,0)$ to either $(R_t(x),t)$ 
or $(L_t(x),t)$, so that $(x,0) \leftrightarrow (y,t)$ as well.
\epr

If $\xi_0 = \mathbbm{1}_x$, then $R_t(x)$ and $L_t(x)$ are, respectively, the right-most 
and the left-most infections present at time $t$. In particular, in this case the 
infection survives for all times if and only if $R_t(x) - L_t(x) \geq 0$ for all 
$t \geq 0$. For $\lambda \in (\lambda_c,\infty)$ it is well known that, given 
$\xi_0=\mathbbm{1}_{0}$, the infection survives with positive probability and there 
exists a constant $\iota = \iota(\lambda) > 0$ such that, conditionally on survival, 
\begin{equation}
\label{convR0}
\lim_{t \to \infty} t^{-1}R_t(0) = \iota \quad \xi\text{-a.s.}
\end{equation}


\subsection{Random walk on top of contact process}
\label{subsec:RW}

Under assumptions (\ref{condjumprates1}--\ref{condjumprates2}), the random walk $W$ 
can be constructed as follows. Let $N:=(N_t)_{t \geq 0}$ be a Poisson process with rate 
$\gamma$. Denote by $J:=(J_k)_{k \in \N_0}$ its generalized inverse, i.e., $J_0 = 0$ and 
$(J_{k+1}-J_k)_{k \in \N_0}$ are i.i.d.\ $\mathrm{EXP}(\gamma)$ random variables. Let 
$U:=(U_k)_{k\in\N}$ be an i.i.d.\ sequence of $\mathrm{UNIF}([0,1])$ random variables, 
independent of $N$. Set $S_0:= 0$ and, recursively for $k \in \N_0$,
\begin{equation}
\label{defwalkskeleton}
S_{k+1} := S_k + 2 \left( \mathbbm{1}_{\{0 \leq U_{k+1} \leq \alpha_0/\gamma \}} 
+ \xi_{J_{k+1}}(S_k) \mathbbm{1}_{\{\alpha_0/\gamma < U_{k+1} \leq \alpha_1/\gamma\}} 
\right) - 1,
\end{equation}
i.e., $S_{k+1}=S_k+1$ with probability $\alpha_i/\gamma$ and $S_{k+1}=S_k-1$ with 
probability $\beta_i/\gamma=1-\alpha_i/\gamma$ when $\xi_{J_{k+1}}(S_k)=i$, for $i=0,1$ 
(recall that $\alpha_0 < \alpha_1$ by (\ref{condjumprates1}--\ref{condjumprates2})). 
Setting
\begin{equation}
\label{defwalk}
W_t := S_{N_t},
\end{equation}
we can use the right-continuity of $\xi$ to verify that $W$ indeed is a Markov process 
with the correct jump rates.

A useful property of the above construction is that it is monotone in the environment, 
in the following sense. For two dynamic random environments $\xi$ and $\xi^{\prime}$, 
we say that $\xi \leq \xi^{\prime}$ when $\xi_t \leq \xi^{\prime}_t$ for all $t \geq 0$. 
Writing $W = W(\xi)$ in the previous construction (i.e., exhibiting $W$ as a function 
of $\xi$), it is clear from \eqref{defwalkskeleton} that 
\begin{equation}
\label{monotwalk}
\xi \le \xi^{\prime} \quad \Longrightarrow \quad 
W_t(\xi) \leq W_t(\xi^{\prime}) \quad \forall\,t \geq 0.
\end{equation}

We denote by
\begin{equation}
\label{defcurlyG}
\mathcal{G}_t : = \mathcal{F}_t \vee 
\sigma\big((N_s)_{s \in [0,t]}, (U_k)_{1 \le k \le N_t}\big)
\end{equation}
the filtration generated by all the random variables that are used to define the contact 
process $\xi$ and the random walk $W$.


\section{The strong law of large numbers}
\label{sec:LLN}

Theorem~\ref{LLN}(a) is proved in two steps. In Section~\ref{subsec:LLNfull} we use
subadditivity to prove the SLLN when $\xi$ starts from $\delta_{\mathbf{1}}$. In 
Section~\ref{subsec:LLNeq} we show that, under the hypotheses stated about $\mu$ in Theorem~\ref{LLN}(a), 
we can couple two copies of $\xi$ starting from $\mu$ and $\delta_{\mathbf{1}}$ 
so as to transfer the SLLN, with the same speed.

In the following, for a random process $X = (X_t)_{t \in \mathcal{I}}$ with $\mathcal{I} 
= \R$ or $\mathcal{I}=\Z$, we write
\begin{equation}
\label{notationproc}
X_{[0,t]} := (X_s)_{s \in [0,t]\cap \mathcal{I}}.
\end{equation}


\subsection{Starting from the full configuration: subadditivity}
\label{subsec:LLNfull}

Since $\eta \leq \mathbf{1}$ for all $\eta \in \Omega$, it follows from 
\eqref{monotwalk} and Lemma~\ref{monotCP} that $W_t(\xi(\eta)) \leq 
W_t(\xi(\mathbf{1}))$ for all $t \geq 0$. Therefore, if in the graphical 
construction we replace $\xi_s$ by $\mathbf{1}$ at any given time $s$, 
then the new increments after time $s$ lie to the right of the old increments 
after time $s$, and are independent of the increments before time $s$. This 
leads us to a subadditivity argument, which we now formalize.

For $n \in \N_0$, let
\begin{equation}
\label{defHn}
\begin{array}{lllll}
H^{(n)} & = & \big(H^{(n)}_t(x)\big)_{t \geq 0, x \in \Z} 
& := & \big( H_{t+n}(x+W_n) - H_n(x+W_n) \big)_{t \geq 0, x \in \Z},\\
I^{(n)} & = & \big(I^{(n)}_t(x) \big)_{t \geq 0, x \in \Z} 
& := & \big( I_{t+n}(x+W_n) - I_n(x+W_n) \big)_{t \geq 0, x \in \Z},\\
N^{(n)} & = & \big(N^{(n)}_t \big)_{t \geq 0} 
& := & \big( N_{t+n} - N_n \big)_{t \geq 0},\\
U^{(n)} & = & \big( U^{(n)}_k \big)_{k \in \N} 
& := & \big( U_{k+N_n} \big)_{k \in \N}.\\
\end{array}
\end{equation}
Then, for any $n\in\N_0$, $(H^{(n)}, I^{(n)}, N^{(n)}, U^{(n)})$ has the same distribution 
as $(H,I,N,U)$ and is independent of 
\begin{equation}
H^{(j)}_{[0,n-j]}, I^{(j)}_{[0,n-j]}, N^{(j)}_{[0,n-j]}, U^{(j)}_{[1,N^{(j)}_{n-j}]}, 
\qquad 0 \leq j \leq n-1.
\end{equation}

Abbreviate $\xi = \xi(\eta, H,I)$ and $W = W(\xi,N,U)$. For $n\in\N_0$, let
\begin{equation}
\label{defWn}
\begin{array}{lll}
\xi^{(n)} & := & \xi(\mathbf{1}, H^{(n)}, I^{(n)}),\\
W^{(n)} & := & W(\xi^{(n)}, N^{(n)}, U^{(n)}),\\
\end{array}
\end{equation}
and define the double-indexed sequence
\begin{equation}
\label{defsubadseq}
X_{m,n} := W^{(m)}_{n-m}, \qquad n,m \in \N_0, \, n \geq m.
\end{equation}

\begin{lemma}
\label{subad}
The following properties hold:\\
(i) For all $n,m \in \N_0$, $n \geq m$: $X_{0,n} \leq X_{0,m} + X_{m,n}$.\\
(ii) For all $n \in \N_0$: $(X_{n,n+k})_{k \in \N_0}$ has the same distribution 
as $(X_{0,k})_{k \in \N_0}$.\\
(iii) For all $k \in \N$: $(X_{nk,(n+1)k})_{n \in \N_0}$ is i.i.d.\\
(iv) $\sup_{n \in \N} \mathbb{E}_{\delta_{\mathbf{1}}}\left[n^{-1}|X_{0,n}| 
\right] < \infty$.
\end{lemma}

\bpr 
(i) Fix $n,m \in \N_0$, $n \geq m$ and define $\hat\xi := \xi(\hat\eta, H^{(m)},
I^{(m)})$, where $\hat \eta(x) = \xi_m(x+W_m)$. This is the contact process after 
time $m$ as seen from $W_m$. Note that $X_{0,n}-X_{0,m} = W_n - W_m = W_{n-m}
(\hat\xi,N^{(m)},U^{(m)})$. Since $\hat \eta \le \mathbf{1}$, it follows from 
\eqref{monotwalk} and Lemma~\ref{monotCP} that the latter is $\le W_{n-m}
(\xi^{(m)},N^{(m)},U^{(m)}) = W^{(m)}_{n-m}$.

\noindent
(ii) Immediate from the construction.

\noindent 
(iii) By definition, $X_{nk,(n+1)k} = W_k(\xi^{(nk)},N^{(nk)},U^{(nk)})$. By construction, 
for each $t \geq 0$, $W_t(\xi,N,U)$ is a function of $N_{[0,t]}$, $U_{[1, N_t]}$ and 
$\xi_{[0,t]}$, which in turn is a function of $H_{[0,t]},I_{[0,t]}$ and $\eta$. Therefore 
$X_{nk,(n+1)k}$ is equal to a (fixed) function of 
\begin{equation}
H^{(nk)}_{[0,k]},I^{(nk)}_{[0,k]},N^{(nk)}_{[0,k]},
U^{(nk)}_{[1, N^{(nk)}_{(n+1)k}]},
\end{equation} 
which are jointly i.i.d.\ in $n$ (when $k$ is fixed).

\noindent(iv) 
This follows from the fact that $|W_t| \leq N_t$.
\epr

Lemma~\ref{subad} allows us to prove the SLLN when $\xi$ starts from $\delta_{\mathbf{1}}$.

\begin{proposition}\label{LLNfull}
Let
\begin{equation}
\label{defv}
v(\lambda) := \inf_{n \in \N} \mathbb{E}_{\delta_{\mathbf{1}}}\left[n^{-1}W_n \right], \;\; \lambda \in [0,\infty).
\end{equation}
Then
\begin{equation}
\label{eqLLNfull}
\lim_{t \to \infty} t^{-1} W_t = v(\lambda) \quad \mathbb{P}_{\delta_{\mathbf{1}}}
\text{-a.s.\ and in } L^p,\,\,p \geq 1.
\end{equation}
\end{proposition}

\bpr 
Conditions (i)--(iv) in Lemma~\ref{subad} allow us to apply the subbaditive ergodic 
theorem of Liggett~\cite{Li85b} (see also Liggett~\cite{Li85}, Theorem VI.2.6) to 
the sequence $(X_{0,n})_{n \in \N_0} = (W_n)_{n \in \N_0}$, which gives $\lim_{n 
\to \infty}n^{-1} W_n = v$ $\mathbb{P}_{\delta_{\mathbf{1}}}$-a.s. Via a standard argument 
this can subsequently be extended to $(t^{-1}W_t)_{t \geq 0}$ by using that, for 
any $n \in \N_0$,
\begin{equation}
\label{Wdombym}
\sup_{s \in [0,1]}|W_{n+s} -W_{n}| \leq N_{n+1}-N_n,
\end{equation}
which implies that $\lim_{t \to \infty} t^{-1}|W_t-W_{\lfloor t \rfloor}| = 0$ a.s.\ w.r.t.\
$\mathbb{P}_{\delta_{\mathbf{1}}}$. Since $|W_t| \leq N_t$, we see that $(t^{-p}|W_t|^p)_{t \geq 1}$ 
is uniformly integrable for any $p \geq 1$, so the convergence also holds in $L^p$.
\epr


\subsection{Other initial measures: coupling}
\label{subsec:LLNeq}

In this section, we show that whenever two copies of the contact process starting from $\mu$ and 
$\delta_{\mathbf{1}}$ can be coupled so as to agree with large probability at large times inside a 
space-time cone, the LLN holds also under $\mathbb{P}_{\mu}$ with the same velocity $v(\lambda)$.
We subsequently show that such a coupling is possible when $\mu$ is stochastically larger than a 
non-trivial shift-invariant and ergodic measure. Some remarks regarding extensions are made after the 
corresponding results.

For $m>0$, let
\begin{equation}
V_{m} := \big\{(x,s) \in \Z \times [0,\infty) 
\colon\, |x| \le m s \big\}
\end{equation}
be a cone of inclination $m$ opening upwards in space-time.

\begin{proposition}\label{prop:LLNundercoupling}
Fix $\lambda \in (0,\infty)$, and let $\xi^{(\mu)}$ and $\xi^{(\mathbf{1})}$ denote
the contact process started from $\mu$ and $\mathbf{1}$, respectively.
Suppose that there exists a coupling measure $\mathbb{P}$
of $\xi^{(\mu)}$ and $\xi^{(\mathbf{1})}$ such that,
for some $m > |v_0| \vee |v_1|$,
\begin{equation}
\label{eq:sCPcoupled}
\lim_{T \to \infty}\mathbb{P}\left(\exists\,(x,t) \in V_{m} \cap \Z \times [T,\infty)
\colon\,\xi^{(\mu)}_t(x) \neq \xi^{(\mathbf{1})}_t(x)\right) = 0.
\end{equation}
Then
\begin{equation}
\label{eq:LLNundercoupling}
\lim_{t \to \infty} t^{-1}W_t = v(\lambda) \quad \mathbb{P}_{\mu} \text{-a.s.\ and in } L^p, \; p\ge1,
\end{equation}
where $v(\lambda)$ is as in \eqref{defv}.
\end{proposition}
\begin{proof}
Let $N^{(\mu)}$, $U^{(\mu)}$ and $N^{(\mathbf{1})}$, $U^{(\mathbf{1})}$ be independent copies of $N,U$ and
rename $\mathbb{P}$ to denote the joint law of $\xi^{(a)}, N^{(a)}, U^{(a)}$, $a \in \{\mu, \mathbf{1}\}$.
Then $W^{(\mu)} := W(\xi^{(\mu)},N^{(\mu)}, U^{(\mu)})$ under $\mathbb{P}$ has the same law as $W$ 
under $\mathbb{P}_{\mu}$. 

Denote by $\overline{\mathbf{0}}$, $\overline{\mathbf{1}}$ the elements of $\Omega^{[0,\infty)}$ that are constant 
and equal to $0$, respectively, $1$, i.e., $\overline{\mathbf{0}}_t(x) = 0$ for all $(x,t) \in \Z \times [0,\infty)$
and analogously for $\overline{\mathbf{1}}$, and set
\begin{equation}
\label{defoverlineW}
\overline{W}^{(\mu)} := |W(\overline{\mathbf{1}}, N^{(\mu)}, U^{(\mu)})| \vee |W(\overline{\mathbf{0}}, N^{(\mu)}, U^{(\mu)})|.
\end{equation}
For $T > 0$, let
\begin{equation}\label{defD_T}
D_T := \left\{\overline{W}^{(\mu)}_s \le m s \; \;\forall \; s \ge T \right\}
\end{equation}
and
\begin{equation}\label{defGamma_T}
\Gamma_T := D_T \cap \left\{\xi^{(\mu)}_s(x) = \xi^{(\mathbf{1})}_s(x) \;\; \forall \;(x,s) \in V_m \cap \Z \times [T,\infty) \right\}.
\end{equation}
For $i \in \{0,1\}$, $W(\overline{\mathbf{i}},N^{(\mu)},U^{(\mu)})$ is a homogeneous random walk with total jump rate 
$\gamma$ and drift $v_i$. Hence, by \eqref{eq:sCPcoupled},
\begin{equation}
\label{probGammahigh}
\lim_{T \to \infty}\mathbb{P}\left( \Gamma_T \right) = 1.
\end{equation}
Therefore it suffices to prove that
\begin{equation}
\label{eq:sufficientforLLNcoupling}
\mathbb{P} \left( \lim_{t \to \infty} t^{-1}W^{(\mu)}_t = v ~\Big| ~\Gamma_T \right) = 1 \;\; \forall \;\; T >0.
\end{equation}

In order to prove \eqref{eq:sufficientforLLNcoupling}, we couple $W^{(\mu)}$ with a random walk $\widehat{W}$ 
distributed as $W$ under $\mathbb{P}_{\delta_{\mathbf{1}}}$, as follows. Fix $T>0$ and let $\widehat{N} = (\widehat{N}_s)_{s \ge 0}$, 
$\widehat{U} = (\hat{U}_{n})_{n \in \N}$ be defined by
\begin{equation}
\label{defhatN}
\widehat{N}_t := \left\{ \begin{array}{ll}
N^{(\mathbf{1})}_t & \text{ if } t \le T, \\
N^{(\mathbf{1})}_T + N^{(\mu)}_t - N^{(\mu)}_T & \text{ otherwise,}
\end{array}\right.
\end{equation}
and
\begin{equation}
\label{defhatU}
\widehat{U}_n := \left\{ \begin{array}{ll}
U^{(\mathbf{1})}_n & \text{ if } n \le N^{(\mathbf{1})}_T, \\
U^{(\mu)}_n & \text{ otherwise.}
\end{array}\right.
\end{equation}
Then it is clear that $\widehat{W} := W(\xi^{(\mathbf{1})}, \widehat{N}, \widehat{U})$ has the correct law,
and that $\widehat{N}_{[0,T]}, \widehat{U}_{[1,\widehat{N}_T]}$ are independent of $W^{(\mu)}$.
Moreover, we claim that 
\begin{equation}
\label{claimequalityonGammaT}
\text{ on }\Gamma_T\colon\,\quad  W^{(\mu)}_T = \widehat{W}_T \quad \Longrightarrow \quad W^{(\mu)}_s = \widehat{W}_s \; \forall \; s \ge T.
\end{equation}
To see \eqref{claimequalityonGammaT}, note that, by monotonicity,
\begin{equation}
W(\overline{\mathbf{0}}, N^{(\mu)}, U^{(\mu)}) \le W^{(\mu)}_s \le W(\overline{\mathbf{1}}, N^{(\mu)}, U^{(\mu)}),
\end{equation}
so that, on $D_T$, $(W^{(\mu)}_s,s) \in V_m$ for all $s \ge T$.
Since $W^{(\mu)}$ and $\widehat{W}$ use the same jump decisions after time $T$,  if they are equal at time $T$ and $\Gamma_T$ 
occurs, then they will see forever the same random environment, and will thus remain equal for all subsequent times.

With this observation, we are now ready to prove \eqref{eq:sufficientforLLNcoupling} by showing that
\begin{equation}
\label{eq:sufficientforcoupling2}
\mathbb{P} \left( \lim_{t \to \infty} t^{-1}W^{(\mu)}_t = v ~\Big| ~\Gamma_T, W^{(\mu)}_T = x \right) = 1
\end{equation}
for each $x \in \Z \cap [-mT,mT]$. To that end, first note that, for each fixed $x$, there exists an event $B_x \in \sigma(\widehat{N}_{[0,T]}, 
\widehat{U}_{[1,\widehat{N}_T]})$ with positive probability such that $\widehat{W}_T = x$ on $B_x$. Indeed, since all the jump rates 
$\alpha_i,\beta_i$, $i \in \{0,1\}$ are strictly positive, we can fix the number and direction of jumps of $\widehat{W}$ on $[0,T]$ by imposing 
restrictions on $\widehat{N}_{[0,T]}, \widehat{U}_{[1,\widehat{N}_T]}$. To conclude \eqref{eq:sufficientforcoupling2}, we write
\begin{equation}
\begin{aligned}
\mathbb{P} \left( \lim_{t\to \infty} t^{-1} W^{(\mu)}_t = v ~\Big| ~\Gamma_T, W^{(\mu)}_T = x \right) 
& = \mathbb{P} \left( \lim_{t\to \infty} t^{-1} W^{(\mu)}_t = v ~\Big| ~\Gamma_T, W^{(\mu)}_T = x, B_x \right)  \\
& = \mathbb{P} \left( \lim_{t\to \infty} t^{-1} W^{(\mathbf{1})}_t = v ~\Big| ~\Gamma_T, W^{(\mu)}_T = x, B_x \right)  \\
& = 1,
\end{aligned}
\end{equation}
where we use Proposition~\ref{LLNfull}.
\end{proof}

\begin{remark}
\label{rem:nonellipticrates}
\emph{
In the case $\mu = \nu_\lambda$ (for which \eqref{eq:sCPcoupled} holds by Proposition~\ref{prop:sCPcoupled_ergodic} below), 
the conclusion of Proposition~\ref{prop:LLNundercoupling} is true even when some of the jump rates $\alpha_i,\beta_i$, $i \in \{0,1\}$ 
are equal to zero (note that the proof of Proposition~\ref{LLNfull} does not need all rates to be strictly positive). To adapt the proof, 
replace the conditioning on $W_T^{(\mu)}=x$ in \eqref{eq:sufficientforcoupling2} by the event $\{N^{(\mu)}_T = 0\}$, which implies 
$W_T^{(\mu)} = 0$. Then $(W^{(\mu)}_{t+T}-W^{(\mu)}_T)_{t \ge 0}$ under $\mathbb{P}(\cdot | N^{(\mu)}_T = 0)$ has the same
distribution as $W$ under $\mathbb{P}_{\nu_\lambda}$. Since $\widehat{N}_T = 0$ implies $\widehat{W}_T = 0$ and has positive
probability, the claim follows as before.}
\end{remark}

We next show that \eqref{eq:sCPcoupled} is satisfied whenever $\mu$ contains a non-trivial ergodic measure. 
This together with Proposition~\ref{prop:LLNundercoupling} finishes the proof of Theorem~\ref{LLN}(a).

\begin{proposition}
\label{prop:sCPcoupled_ergodic}
If $\mu$ is stochastically larger than a non-trivial probability measure $\bar{\mu}$ that is shift-invariant and ergodic,
then \eqref{eq:sCPcoupled} holds under the coupling given by the graphical representation in Section~{\rm \ref{subsec:cp}}.
\end{proposition}

\begin{proof}
Let $\xi^{(\mu)}$, $\xi^{(\bar{\mu})}$ and $\xi^{(\mathbf{1})}$ be copies of the contact process started from the
corresponding initial measures, constructed with the same graphical representation given by $H,I$ and such 
that $\xi_0^{(\bar{\mu})} \le \xi_0^{(\mu)}$. Since this coupling preserves the ordering, we have
 $\xi^{(\bar{\mu})}_t \le \xi^{(\mu)}_t \le \xi^{(\mathbf{1})}_t$ for all $t \ge 0$, and so we may assume
 that $\mu$ is non-trivial, shift-invariant and ergodic.

Denote by $P$ the joint law of $\xi_0^{(\mu)}$, $H$, $I$. Regarding $P$ as a law on the product space 
\begin{equation}
\left( \{0,1\} \times D(\N_0, [0,\infty))^2 \right)^{\Z} 
=  \{0,1\}^{\Z} \times \left(D(\N_0, [0,\infty))^2 \right)^{\Z},
\end{equation}
where $D(\N_0, [0,\infty))$ is the space of c\`adl\`ag functions from $[0,\infty)$ to $\N_0$,
we see that $P$ is shift-ergodic because it is the product of probability measures that are shift-ergodic, namely, 
$\mu$ and the distributions of $H$ and $I$. Let
\begin{equation}
\label{defconesurvival}
\Lambda_x := \big\{ \eta(x) = 1,\, \left(x-L_t(x)\right) \wedge 
\left(R_t(x)-x \right) \ge \lfloor (\iota/2)t \rfloor  \,\,\forall\,t \geq 0\big\},
\end{equation}
i.e., the event that $x$ generates a ``wide-spread infection'' (moving at speed at least
half the typical asymptotic speed $\iota$). Since $\Lambda_x$ is a translation of 
$\Lambda_0$, we have
\begin{equation}
\label{avwsinf}
\lim_{n \to \infty}\frac{1}{n}\sum_{x=1}^n \mathbbm{1}_{\Lambda_x} 
= P(\Lambda_0) = P\left(\xi_0(0)=1\right)P\left(\Lambda_0 \mid \xi_0(0) = 1\right) := \varrho > 0 \;\; P \text{-a.s.},
\end{equation}
where the last inequality is justified as follows: $P(\xi_0(0) = 1)>0$ since $\mu$ is assumed to be non-trivial,
and $P(\Lambda_0 \mid \xi_0(0) = 1) > 0$ by \eqref{convR0} and local modifications of the 
graphical representation.

Next, for $n \in \N$, define $Z_n \in \N$ by the equation
\begin{equation}
\label{defLn}
\sum_{x=1}^{Z_n}\mathbbm{1}_{\Lambda_x} = n.
\end{equation}
Then we also have
\begin{equation}
\label{convLn}
\lim_{n \to \infty}\frac{Z_n}{n} = \varrho^{-1} \qquad P \text{-a.s.}
\end{equation}
$(Z_n)_{n \in \N}$ marks the positions of wide-spread infections to the right 
of the origin, i.e., $x>0$ such that $\Lambda_x$ occurs. Equation \eqref{convLn} means 
that these wide-spread infections are not too far apart. Extending the definition of 
$Z_n$ to the negative integers, we obtain analogously that $\lim_{n \to \infty}n^{-1}
(-Z_{-n}) = \varrho^{-1}$ $P$-a.s. Let $\mathcal{Z} := \cup_{n \in \N} \{Z_n,Z_{-n}\}$ and
\begin{equation}
\label{defSaferegion} 
\mathcal{S} := \big\{(y,t) \in \Z \times [2/\iota,\infty) \colon\,\exists\,x \in \mathcal{Z} 
\text{ such that } |y-x| \leq (\iota/2)t -1 \big\}.
\end{equation}
Then $\mathcal{S}$ is the union of cones of inclination angle $\iota/2$ with tips at 
$(2/\iota, z)$ with $z \in \mathcal{Z}$ (see Fig.~\ref{fig-saferegion}). We call 
$\mathcal{S}$ the \emph{safe region}. This is justified by the following fact, whose 
proof is a direct consequence of Lemma~\ref{depinclusters}.

\begin{lemma}
\label{Sissafe}
If $(x,t) \in \mathcal{S}$, then $\xi^{(\mu)}_t(x) = \xi^{(\mathbf{1})}_t(x)$.
\end{lemma}

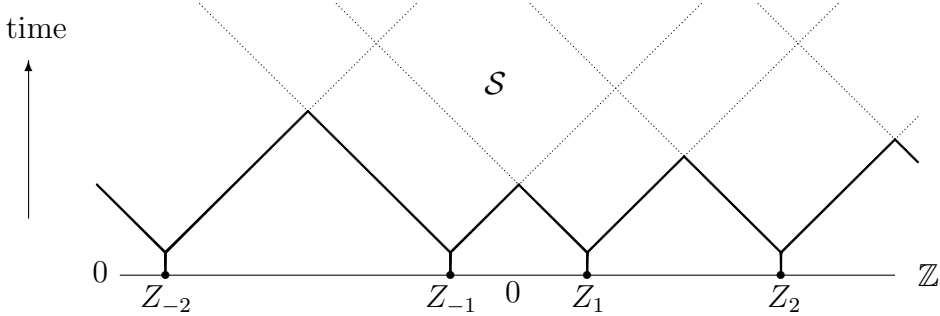
\begin{figure}[hbtp]
\vspace{1.5cm}
\begin{center}
\setlength{\unitlength}{0.3cm}
\begin{picture}(32,8)(0,-1)
\put(0,-1){\line(1,0){34}}
\qbezier[100](2,0)(7,5)(13,11)
\qbezier[100](14.5,0)(9.5,5)(3.5,11)
\qbezier[100](14.5,0)(19.5,5)(25.5,11)
\qbezier[100](20.5,0)(15.5,5)(9.5,11)
\qbezier[100](20.5,0)(25.5,5)(31.5,11)
\qbezier[100](29,0)(24,5)(18,11)
\qbezier[54](29,0)(32,3)(35,6)
\qbezier[54](34,5)(31,8)(28,11)
{\thicklines
\qbezier(2,0)(1,1)(-1,3)
\qbezier(2,0)(3,1)(8.25,6.25)
\qbezier(14.5,0)(11.5,3)(8.25,6.25)
\qbezier(14.5,0)(15.5,1)(17.5,3)
\qbezier(20.5,0)(19.5,1)(17.5,3)
\qbezier(20.5,0)(22.5,2)(24.75,4.25)
\qbezier(29,0)(27,2)(24.75,4.25)
\qbezier(29,0)(31,2)(34,5)
\qbezier(34,5)(34.5,4.5)(35,4)
\qbezier(2,-1)(2,0)(2,0)
\qbezier(14.5,-1)(14.5,0)(14.5,0)
\qbezier(20.5,-1)(20.5,0)(20.5,0)
\qbezier(29,-1)(29,0)(29,0)
}
\put(2,-1){\circle*{.35}} 
\put(14.5,-1){\circle*{.35}}
\put(20.5,-1){\circle*{.35}} 
\put(29,-1){\circle*{.35}}
\put(0.9,-2.4){$Z_{-2}$} 
\put(13.4,-2.4){$Z_{-1}$}
\put(16.9,-2.2){$0$}
\put(19.8,-2.4){$Z_1$} 
\put(28.4,-2.4){$Z_2$}
\put(-1.2,-1.3){$0$}
\put(-4,1.5){\vector(0,1){7}}
\put(-5,9.5){\text{time}}
\put(35,-1.5){$\mathbb{Z}$}
\put(16,7){$\mathcal{S}$}
\end{picture}
\end{center}
\caption{\small Cones have inclination angle $\iota/2$. The safe region $\mathcal{S}$
lies above the thick lines.} 
\label{fig-saferegion}
\end{figure}

By Lemma~\ref{Sissafe}, it is enough to prove that, for any $m>0$,
\begin{equation}
\label{finiteintersection}
V_m \cap \mathcal{S}^c \text{ is a bounded subset of } 
\Z \times [0,\infty) \quad P \text{-a.s.}
\end{equation}
To that end, note that $\mathcal{S}^c$ is contained in the union of space-time ``houses'' (unions of 
triangles and rectangles) with base at time $0$. The tips of the houses to the right 
of $0$ form a sequence with spatial coordinates $\tfrac12(Z_{n+1}+Z_n)$ and temporal 
coordinates $(Z_{n+1}-Z_n+2)/\iota$, $n \in \N$. By \eqref{convLn}, the ratio of 
temporal/spatial coordinates tends to $0$ as $n \to \infty$, so that only finitely 
many tips can be inside $V_m$. The same is true for the tips of the houses 
to the left of $0$. Therefore $V_m$ touches only finitely many houses, which 
proves \eqref{finiteintersection}. 
\end{proof}

\begin{remark}
\label{rem:initconfigposdensity}
\emph{
It is possible to show that \eqref{eq:sCPcoupled} holds for any initial configuration that has a 
positive lower density of infections to the right and to the left of the origin.  We will not pursue 
this extension here, and content ourselves with giving a sketch of a proof strategy that uses 
the techniques from Section~\ref{sec:moreCP} below. Let
\begin{equation}
\begin{aligned}
\Lambda^*_x := \Big\{ 
& \text{there exist two paths } \pi^-_s, \pi^+_s \text{ such that } (x,0) \leftrightarrow (\pi^{\pm}_s,s)\\ 
& \text{and } \lfloor \iota/2 s \rfloor \le \pm (\pi^{\pm}_s - x)  \le 2\iota s\;\,\forall\;s \ge 0 \Big\}.
\end{aligned}
\end{equation}
Fix $x$ with $\xi_0(x) = 1$.  With the help of the methods used in the proof of Lemma~\ref{lemma:momZdelta}
below, we may show that if $\Lambda^*_x$ does not occur, then there is a positive random variable $L_x$ with a (uniform) exponential moment such that, for any $k \in \N$, 
$\Lambda^*_{x+L_x+k}$ is independent of $\Lambda^*_x,L_x$ and distributed as $\Lambda^*_0$. 
Therefore, after a geometric number of trials 
we find a point $y>x$ such that $\xi_0(y) = 1$ and $\Lambda^*_y$ occurs. Next, we use Lemma~\ref{lemma:stochdomleftpi} 
and the FKG-inequality (see \cite{BeGr91}) to forget all information gathered so far and start afresh at the next point $z>y$ 
such that $\xi_0(z) = 1$. In this way we prove that the configuration $\eta^*(x) := \xi_0(x) \mathbbm{1}_{\Lambda^*_x}$
has a positive lower density, and from this point on we may continue as in the proof of Proposition~\ref{prop:sCPcoupled_ergodic}.
}
\end{remark}


\section{More on the contact process}
\label{sec:moreCP}

In this section we collect some additional facts about the contact process on $\Z$
that will be needed in the remainder of the paper. The proofs rely on geometric 
observations that will also illuminate the proof strategies developed in 
Sections~\ref{sec:speed}--\ref{sec:regeneration}.

In the following we will use the notation
\begin{equation}\label{notZlex}
\Z_{\le x} := \Z \cap (\infty, x]
\end{equation}
and analogously for $\Z_{\ge x}$.

\paragraph{Stochastic domination.}
We start with a useful alternative construction of the equilibrium $\nu_\lambda$. 
Let $\eta(x) := \mathbbm{1}_{\{C_t(x) \neq \emptyset \; \forall \; t \ge 0\}}$. Then,
by the graphical representation, $\eta$ has distribution $\nu_\lambda$. This follows 
from duality (see Liggett~\cite{Li85}, Chapter VI). We can also graphically construct 
the contact process starting from $\nu_\lambda$: extend the graphical representation 
to negative times, and declare $\xi_t(x) = 1$ if and only if for all $0 \leq s \leq t$ 
there exists a $y$ such that $(y,s) \leftrightarrow (x,t)$, i.e., if and only if there 
exists an infinite infection path going backwards in time from $(x,t)$.

Let $\bar \nu_\lambda$ denote the restriction of $\nu_\lambda$ to $\Z_{\le-1}$. 
Abusing notation, we will write the same symbol to denote the measure on $\Omega$ 
that is the product of $\bar \nu_\lambda$ with the measure concentrated on all sites 
healthy to the right of $-1$. Using the alternative construction above, we can prove 
that the restriction of $\nu_\lambda(\cdot \mid \eta(0) = 1)$ to $\Z_{\le -1}$ is 
stochastically larger than $\bar\nu_\lambda$. In the following, we will focus on a 
similar result for the distribution of $\xi_t$ to the left of certain infection paths.

For $\varpi_{[0,t]}$ a nearest-neighbor c\`adl\`ag path with values in $\Z$, let
\begin{align}
\label{deffiltvarpi}
\bar{\mathcal{R}}^{\varpi}_t := \sigma \Big\{ & \big(\xi_0(x)\big)_{x \ge \varpi_0}, 
\nonumber\\
&\big(H_v(x)-H_u(x), I_v(x)-I_u(x)\big)_{\{(x,u,v) \in \Z \times [0,t]^2\colon\, 
u \le v,\,\, x \ge \sup_{s \in [u,v]}\varpi_s\} }\Big\}.
\end{align}
Suppose that $\pi_{[0,t]}$ is a random path of the same type, with the following 
properties:
\begin{enumerate}
\item[(p1)] 
$\xi_0(\pi_0)=1$ a.s.\ and $(\pi_s,s) \leftrightarrow (\pi_u,u)$ for all $s,u \in [0,t]$.
\item[(p2)] 
$\pi$ is $\mathcal{F}$-adapted and $\{\pi_s \ge \varpi_s \; \forall \; s \in [0,t]\} \in 
\bar{\mathcal{R}}^\varpi_t$ for all deterministic paths $\varpi$.
\end{enumerate}
We call $\pi$ a \emph{random infection path} (see Fig.~\ref{fig:infecpath}), a name 
that is justified by (p1). Property (p2) means that $\pi$ is causal and that, when 
we discover it, we leave the graphical representation to its left untouched. For 
such $\pi$, let
\begin{align}
\label{deffiltpi}
\mathcal{R}^{\pi}_t := \Big\{ A \in \mathcal{F}_{\infty}\colon\, 
&A \cap \{\pi_s \ge \varpi_s \; \forall \; s \in [0,t]\} \in 
\bar{\mathcal{R}}^{\varpi}_t \text{ for every deterministic} \nonumber\\ 
&\text{nearest-neighbor c\`adl\`ag path } \varpi_{[0,t]}\Big\}. 
\end{align}
Note that, since $\pi$ is an infection path, also $(\xi_s(x))_{x \ge \pi_s} \in 
\mathcal{R}^{\pi}_t$ for each $s \in [0,t]$ (see the proof of Lemma~\ref{depinclusters}).
We have the following stochastic domination result.

\begin{lemma}
\label{lemma:stochdomleftpi}
For any random infection path $\pi_{[0,t]}$ as above, the law of $\xi_t(\cdot+\pi_t+1)$ 
under $\mathbb{P}_{\bar\nu_{\lambda}}(\cdot \mid \mathcal{R}^{\pi}_t)$ is stochastically 
larger than $\bar \nu_{\lambda}$.
\end{lemma}

\begin{proof}
Construct $\mathbb{P}_{\bar\nu_{\lambda}}$ from a graphical representation on $\Z \times \R$ 
as outlined above by adding healing events on $(x,0)$ for each $x \in \Z_{\ge 0}$.
Extend $\pi$ to negative times by making it equal to the right-most infinite infection path 
going backwards in time from $(\pi_0,0)$. (Such a path exists because $\xi_0(\pi_0) = 1$.) 
We may check that the resulting path still has properties (p1) and (p2). Extend also 
$\mathcal{R}^{\pi}_t$ to include negative times.

Next, regard $H$ and $I$ as Poisson point processes on subsets of $\Z \times \R$. Let
(see Fig.~\ref{fig:infecpath})
\begin{equation}
\label{pstochdomleftpieq1}
D := \big\{(x,s) \in \Z \times \R \colon\, s > t \text{ or } \pi_s > x\big\}.
\end{equation}
\vspace{-0.5cm}

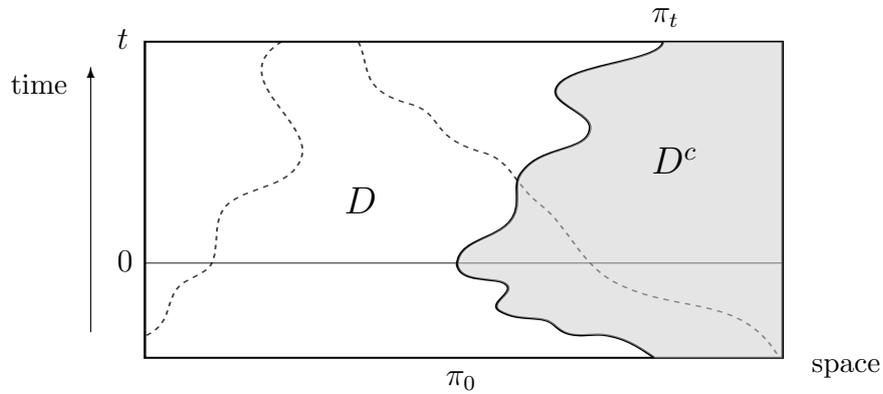
\begin{figure}[htb]
\begin{picture}(0.5,0.5)
\put(110,0){\includegraphics[height= 120pt,width=240pt]{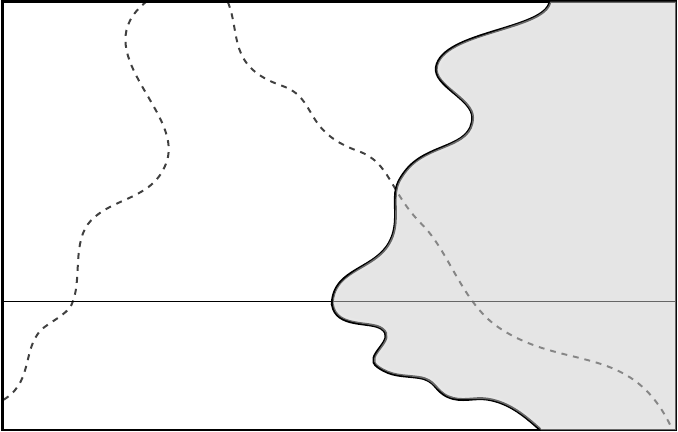}}
\end{picture}
\put(90,10){\vector(0,1){100}}
\put(60,100){\small\mbox{time}}
\put(360,-5){\small\mbox{space}}
\put(223,-10){$\pi_0$}
\put(300,127){$\pi_t$}
\put(100,117){$t$}
\put(100,33){$0$}
\put(185,55){\large\mbox{$D$}}
\put(300,70){\large\mbox{$D^c$}}
\caption{\small The thick line represents the random infection path $\pi$. 
The dashed lines represent other infection paths.} 
\label{fig:infecpath}
\end{figure}

\noindent
Given $\mathcal{R}^{\pi}_t$, by (p2) $H$ and $I$ are still Poisson point processes 
with the same densities on $D$. This can be justified first for $\pi$ taking values 
in a countable set and then for general $\pi$ using right-continuity.

With this observation we can couple $\mathbb{P}_{ \nu_\lambda}$ to $\mathbb{P}_{
\bar \nu_\lambda}(\cdot \mid \mathcal{R}^{\pi}_t)$ in the following way. Draw 
independent Poisson point processes $\hat H$, $\hat I$ on $D^c$. Take $\hat \xi$ to 
be the contact process obtained by using $H$, $I$ on $D$ and $\hat H$, $\hat I$ on 
$D^c$. Then $\hat \xi$ is distributed as the contact process under $\mathbb{P}_{
\nu_\lambda}$, and is independent of $\mathcal{R}^{\pi}_t$. Furthermore, 
$\xi_t(x) \ge \hat \xi_t(x)$ for all $x < \pi_t$. Indeed, if $\xi_t(x)=1$, then 
infinite infection paths going backwards in time must either stay inside $D$ or 
cross $\pi$, so that, by (p1), $\xi_t(x) = 1$ as well.
\end{proof}

\begin{remark}
\label{remark:stoptimes}
In Lemma~{\rm \ref{lemma:stochdomleftpi}}, we may replace $t$ by a finite stopping time 
$\mathcal{T}$ w.r.t.\ the filtration $\mathcal{F}$, as long as the event in {\rm (p2)} is 
replaced by $\{\mathcal{T} \le t, \pi_s \ge \varpi_s \; \forall \; s \in [0,\mathcal{T}]\}$ 
and we add $\mathcal{T}$ to $\mathcal{R}^{\pi}_{\mathcal{T}}$. We may also enlarge all 
filtrations by adding information that is independent of $\xi_0,H,I$, in particular, 
$N_{[0,t]}$ and $U_{[1, N_t]}$ (recall Section~{\rm \ref{subsec:RW}}).
\end{remark}

\paragraph{Infection range.}
Lemma~\ref{lemma:momZdelta} below concerns the positions of wide-spread infections. For 
$\delta \in (0, \iota)$ and $x \in \Z$, let $\mathcal{W}^{\delta}_x := \{(z,t) \in \Z 
\times [0,\infty) \colon\, (\iota-\delta)t-1 < z-x \le (\iota + \delta)t \}$ be a wedge 
between two lines of inclination $\iota - \delta$ and $\iota+\delta$. Set $C^{\delta}_t(x) 
:= \{y \in \Z \colon\ (y,t) \leftrightarrow (x,0) \text{ via a path contained in } 
\mathcal{W}^{\delta}_x\}$,
and
\begin{equation}
\label{defZdelta}
Z_{\delta}(x) := \sup \big\{z \in \Z_{< x} \colon\, \xi_0(z) = 1, 
C^{\delta}_t(z) \neq \emptyset \; \forall \; t \ge 0 \big\},
\end{equation}
i.e., the first infected site to the left of $x$ that spreads its infection forever inside 
a wedge.

\begin{lemma}
\label{lemma:momZdelta}
If $\lambda \in (\lambda_c,\infty)$ then $|Z_{\delta}(x)-x|$ has exponential moments 
under $\mathbb{P}_{\bar \nu_\lambda}$ for every $\delta \in (0,\iota)$, uniformly in 
$x \in \Z_{\le 0}$.
\end{lemma}

\begin{proof}
We will use the fact that, for any $\lambda \in (\lambda_c,\infty)$, $\nu_\lambda$ 
stochastically dominates a non-trivial Bernoulli product measure $\mu_\lambda$. This 
follows from Liggett and Steif~\cite{LiSt05}, Theorem 1.2, Durrett and 
Schonmann~\cite{DuSc88}, Theorem 1, and van den Berg, H\"aggstr\"om and 
Kahn~\cite{vdBeHaKa06}, Theorem 3.5. Since $Z_{\delta}(x)$ is monotone in $\xi_0$, 
it is therefore enough to prove the statement under $\mathbb{P}_{\mu_{\lambda}}$. 
We may also assume $x=0$, as $Z_{\delta}(x)$ does not depend on $(\xi_0(z))_{z \ge x}$. 

Construct a sequence of pairs $(Z_n,T_n)_{n\in\N_0}$ as follows. Set $Z_0 = T_0 :=0$ 
and, recursively for $n\in\N_0$,
\begin{equation}
\label{pmomZdeltaeq1}
\begin{array}{rcll}
Z_{n+1} 
&:=& \left\{ \begin{array}{l} Z_n \\
\sup\{z < Z_n - \lceil (\iota + \delta) T_n \rceil \colon\, \xi_0(z) = 1\} 
\end{array} 
\right. 
&\begin{array}{l} \text{ if } T_n = \infty,\\ \text{ otherwise,} \end{array} \\
T_{n+1} 
&:=& \left\{ \begin{array}{l} \infty \\
\inf\{t>0 \colon\, C^{\delta}_t(Z_{n+1}) = \emptyset \} 
\end{array}
\right.
&\begin{array}{l}\text{ if } T_n = \infty,\\ \text{ otherwise.} \end{array}
\end{array}
\end{equation}
Conditionally on $T_n < \infty$, $\Delta_{n+1}:=Z_{n+1}-Z_{n}+\lceil(\iota+\delta)T_n \rceil$ 
and $T_{n+1}$ are independent of $(Z_k,T_k)_{k=1}^n$ and distributed as $(Z_1$,$T_1)$. This 
is because the region of the graphical representation plus initial configuration on which 
$T_{n+1}$ and $\Delta_{n+1}$ depend is disjoint from the region on which the previous random
variables depend. Since $\mu_\lambda$ is a non-trivial product measure, $|Z_1|$ has 
exponential moments. Noting that $T_1$ is independent of $Z_1$ we conclude, using standard 
facts about the contact process (see Liggett~\cite{Li85}, Chapter VI, Theorem 2.2, Corollary 
3.22 and Theorem 3.23), that $\mathbb{P}_{\mu_\lambda}(T_1 = \infty)>0$ and that, conditionally 
on $T_1< \infty$, $T_1$ has exponential moments. Defining the random index
\begin{equation}
\label{pmomZdeltaeq2}
K := \inf\{n \in \N \colon\, T_n = \infty\}
\end{equation}
whose distribution is $\mathrm{GEO}(\mathbb{P}_{\mu_\lambda}(T_1 = \infty))$, we see that 
$|Z_{\delta}(0)| \le |Z_K|$. Taking $a>0$ such that $\mathbb{E}_{\mu_\lambda}[e^{a(|Z_1| 
+ \lceil (\iota + \delta)T_1\rceil)} \mid T_1 < \infty] < 1/\mathbb{P}_{\mu}(T_1 < \infty)$, 
we get after a short calculation that $\mathbb{E}_{\mu_\lambda}[\mathbbm{1}_{\{K=n\}}
e^{a|Z_n|}]$ decays exponentially in $n$.
\end{proof}


\section{Properties of the speed}
\label{sec:speed}

In this section we prove Theorem~\ref{LLN}(b).

Take $\lambda_{\infty}, \lambda_k \in (\lambda_c,\infty)$ such that 
$\lim_{k \to \infty} \lambda_k = \lambda_{\infty}$, and write 
$\mathbb{P}^{\lambda_\infty}_\mu$, $\mathbb{P}^{\lambda_k}_\mu$ for 
the measure described in Section~\ref{sec:construc} with the corresponding 
parameter and initial measure $\mu$. 
Fix $n \in \N$ and $\epsilon > 0$, and take $L_{n,\epsilon} > 0$ 
such that $\mathbb{E}^{\lambda_k}_{\delta_{\mathbf{1}}}[ N_n 
\mathbbm{1}_{\{N_n > L_{n,\epsilon}\}}] \le \epsilon$ uniformly in 
$k \in \N \cup \{\infty\}$. On the event $\{N_n \le L_{n,\epsilon} \}$, 
$W_n$ depends on $\xi$ only inside the finite space-time region 
$\Z \cap [-L_{n,\epsilon},L_{n,\epsilon}] \times [0,n]$. Therefore
$\mathbb{E}^{\lambda_k}_{\delta_{\mathbf{1}}}[n^{-1}W_n \mathbbm{1}_{\{N_n 
\le L_{n,\epsilon}\}}]$ converges as $k \to \infty$ to the same 
quantity with parameter $\lambda_{\infty}$ (see Liggett~\cite{Li99}, Part I). 
Since $|W_n| \le N_n$, it follows that
\begin{align}
\label{eq:contEWn/n}
\left|\mathbb{E}^{\lambda_k}_{\delta_{\mathbf{1}}}[n^{-1}W_n] 
- \mathbb{E}^{\lambda_{\infty}}_{\delta_{\mathbf{1}}}[n^{-1}W_n]\right| 
\le
& \left|\mathbb{E}^{\lambda_k}_{\delta_{\mathbf{1}}}[n^{-1}W_n 
\mathbbm{1}_{\{N_n \le L_{n,\epsilon}\}}] 
- \mathbb{E}^{\lambda_{\infty}}_{\delta_{\mathbf{1}}}[n^{-1}W_n 
\mathbbm{1}_{\{N_n \le L_{n,\epsilon}\}}]\right| \nonumber \\
& \qquad + 2 \epsilon.
\end{align}
Taking the $\limsup$ as $k \to \infty$ of \eqref{eq:contEWn/n} followed by 
$\epsilon \downarrow 0$, we get that $\lambda \mapsto \mathbb{E}_{\delta_{\mathbf{1}}}
[n^{-1}W_n]$ is continuous. By monotonicity, the latter is also non-decreasing, so it 
follows from \eqref{defv} that $\lambda \mapsto v(\lambda)$ is right-continuous 
and non-decreasing. 

It remains to show that $v(\lambda) \in (v_0, v_1)$ and $\lim_{\lambda\to\infty} v(\lambda)
=v_1$. This will be done in Sections~\ref{subsec:proofvsmall}--\ref{subsec:proofvlarge} below.
These properties come from the fact that the random walk spends positive fractions of its 
time on top of infected sites and on top of healthy sites. To keep track of this, define 
$N^i_t := \#\{n \in \N \colon\, \xi_{J_n}(W_{J_{n-1}})=i\}$, $i \in \{0,1\}$. Recalling the construction of $W$ in Section~\ref{subsec:RW}, we 
may write
\be{repW}
W_t = S^0_{N^0_t}+S^1_{N^1_t},
\ee
where $S^i_n$, $i=0,1$, are discrete-time homogeneous random walks that jump to the right 
with probability $\alpha_i/\gamma$ and to the left with probability $\beta_i/\gamma$. From 
this representation we immediately get the following. 

\begin{lemma}
\label{lemma:speedvsdensity}
\be{eq:speedvsdensity}
\begin{array}{lll}
\liminf\limits_{t\to\infty} t^{-1}W_t = v_0 + (v_1 - v_0)
\liminf\limits_{t \to \infty} (\gamma t)^{-1}N^1_t,\\[0.2cm]
\limsup\limits_{t\to\infty} t^{-1}W_t = v_1 - (v_1 - v_0)
\liminf\limits_{t \to \infty} (\gamma t)^{-1}N^0_t.
\end{array}
\ee
\end{lemma}

Lemma~\ref{lemma:speedvsdensity} is valid for any dynamic random environment, even without 
a SLLN for $W$. But \eqref{eq:speedvsdensity} shows that a SLLN for $W$ holds with speed $v$ 
if and only if a SLLN holds for $N^1$ with limit $\gamma \rho_{\mathrm{eff}}$, where 
$\rho_{\mathrm{eff}} := (v - v_0)/(v_1 - v_0)$ is the effective density of $1$'s seen by 
$W$. Thus, $v > v_0$ and $v<v_1$ are equivalent to, respectively, $\rho_{\mathrm{eff}}>0$ 
and $\rho_{\mathrm{eff}}<1$.


\subsection{Proof of $v(\lambda) < v_1$}
\label{subsec:proofvsmall}

In the contact process, infected sites heal spontaneously. Therefore it is easier to find 
$0$'s than $1$'s. For this reason, it is easier to prove that $W$ often jumps from healthy 
sites than from infected sites.

\begin{proof}
For $k \in \N$, let $Y_k := \xi_{J_k}(W_{J_{k-1}})$, and note that $\{Y_{k+1} = 0\}$ contains 
all configurations that between times $J_k$ and $J_{k+1}$ have a cross at site $W_{J_k}$ and 
no arrows between $W_{J_k}$ and its nearest-neighbors, i.e., such that the events 
$H_{J_{k+1}}(W_{J_k})-H_{J_k}(W_{J_k}) \ge 1$ and $I_{J_{k+1}}(W_{J_k})-I_{J_k}(W_{J_k}) 
= I_{J_{k+1}}(W_{J_k}-1)-I_{J_k}(W_{J_k}-1) = 0$ occur. The probability of the latter events 
given $\sigma\{(J_k, \xi_s, W_s)_{0 \le s \le J_k} \}$ is constant in $k$ and equal to 
$p:=\gamma/(\gamma+2\lambda)(1+\gamma+2\lambda)$. Therefore the sequence $(Y_k)_{k\in\N}$ 
is stochastically dominated by a sequence of i.i.d.\ $\mathrm{BERN}(1-p)$ random variables, 
which implies that $\liminf_{t\to\infty}t^{-1}N^0_t \ge \gamma p > 0$, so that $v(\lambda) 
< v_1$ by Lemma \ref{lemma:speedvsdensity}.
\end{proof}


\subsection{Proof of $v(\lambda) > v_0$ and $\lim_{\lambda \to \infty}v(\lambda) = v_1$}
\label{subsec:proofvlarge}

This is the harder part of the proof. We will need results from Section~\ref{sec:moreCP}. 
In the following we will assume that $v_0 \le 0$. The case $v_0 > 0$ can be treated 
analogously.

Let us start with an informal description of the argument. The idea is that there 
are ``waves of infection'' coming from $\pm \infty$ from which the random walk cannot 
escape. When $v_0 \le 0$, we can concentrate on the waves coming from the left, 
represented schematically in Fig.~\ref{fig:infecwaves}. Each time the random walk 
hits a new wave, there is an infection path starting from its current location 
and going backwards in time entirely to the left of the random walk path. By 
Lemma~\ref{lemma:stochdomleftpi}, at this time the law of $\xi$ to the left of the 
random walk has an appreciable density, which means that there are new waves coming 
in from locations not very far to the left. On the other hand, any infections to the 
right of the random walk can be ignored, since they only push it to the right. But 
doing so makes the random walk behave as a homogeneous random walk with a non-positive 
drift, meaning that it does not take the random walk long to hit the next infection 
wave. Since at each collision there is a fixed probability for the random walk 
to jump while sitting on an infection, $v(\lambda) > v_0$ will follow from 
Lemma~\ref{lemma:speedvsdensity}. With some care in the computations we also get 
the limit for large $\lambda$.

\begin{figure}[htb]
\begin{picture}(100,100)
\put(120,0){\includegraphics[height= 100pt,width=180pt]{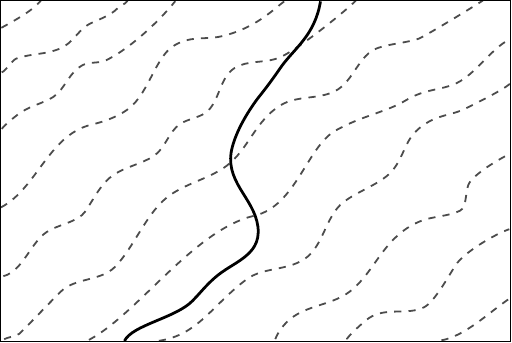}}
\end{picture}
\caption{\small The dashed lines represent infection waves. The thick line 
represents the path of $W$.} 
\label{fig:infecwaves}
\end{figure}

\begin{proof}
Using the graphical representation, we will construct, on a larger probability space, a
second random walk $\hat W$ coupled to $W$ in such a way that $\hat W_t \le W_t$ for all
$t \geq 0$ and that $\hat W$ has a speed with the desired properties. Let
\be{defV}
V_1 := \inf\{t > 0 \colon\, \xi_t(W_t)=1\}.
\ee
Note that $V_1$ has exponential moments under $\mathbb{P}_{\bar{\nu}_{\lambda}}$ by 
Lemma~\ref{lemma:momZdelta} and the fact that $v_0 \le 0$. Let
\be{deftau}
\tau_1 := \inf\big\{t > V_1 \colon\, W_t \neq W_{V_1} \text{ or } 
H_{t}(W_{V_1})>H_{V_1}(W_{V_1})\big\},
\ee
i.e., $\tau_1$ is the first time after time $V_1$ at which either $W$ jumps or there is 
a healing event at the position of the random walk. Note that $\tau_1$ is a stopping time
w.r.t.\ the filtration $\mathcal{G}$ and that, given $\mathcal{G}_{V_1}$, $\tau_1 - V_1$ 
has distribution $\mathrm{EXP}(1+\gamma)$.

We will construct a sequence $(W^{(n)},\tau_n)_{n \in \N}$ with the following properties: 
\begin{enumerate}
\item[(A1)] 
$W^{(n+1)}_t \le W^{(n)}_{\tau_n+t} - W^{(n)}_{\tau_n}$ for all $t \ge 0$;
\item[(A2)] 
$(W^{(n)}, \tau_n)$ is distributed as $\left( W, \tau_1 \right)$ under 
$\mathbb{P}_{\bar \nu_\lambda}$;
\item[(A3)] 
$(W^{(n)}_{[0, \tau_n]}, \tau_n)_{n \in \N}$ is i.i.d.;
\item[(A4)] 
If $\hat v(\lambda) := \mathbb{E}_{\bar{\nu}_{\lambda}}[W_{\tau_1}]/
\mathbb{E}_{\bar{\nu}_{\lambda}}[\tau_1]$, then $\hat{v}(\lambda) > v_0$ 
and $\lim_{\lambda \to \infty} \hat v(\lambda) = v_1$.
\end{enumerate}
Once we have this sequence, we can put $T_0:=0$, $T_n := \sum_{k=1}^{n}\tau_k$ for $n\in\N$,
and
\begin{equation}
\label{defhatW}
\hat{W}_t := \sum_{k=1}^{n} W^{(k)}_{\tau_k} + W^{(n+1)}_{t - T_n} 
\quad \text{ for } \quad T_n \le t < T_{n+1}.
\end{equation}
By (A1), $\hat{W}_t \le W^{(1)}_t$ for all $t \ge 0$. By (A2), the latter is distributed 
as $W$ under $\mathbb{P}_{\bar{\nu}_{\lambda}}$, which by monotonicity is stochastically 
smaller than $W$ under $\mathbb{P}_{\nu_\lambda}$. By (A3), $\lim_{n \to \infty} T_n^{-1}
\hat{W}_{T_n} = \hat{v}(\lambda)$, and so the claim follows from (A4). Thus, it remains to 
construct the sequence $(W^{(n)},\tau_n)_{n\in\N}$ with properties (A1)--(A4). 

To do so, we draw $\xi_0$ from $\bar{\nu}_{\lambda}$, let $\xi^{(1)}:=\xi$, $W^{(1)} := W$, 
define $\tau_1$ as above, and note the following.

\begin{lemma}
\label{claim:stdomafterV}
Under $\mathbb{P}_{\bar\nu_{\lambda}}(\cdot \mid \tau_1, W_{[0,\tau_1]} )$, the law of 
$\xi_{\tau_1}(\cdot+W_{\tau_1})$ is stochastically larger than $\bar\nu_{\lambda}$.
\end{lemma}

\begin{proof}
Since $\xi_{V_1}(W_{V_1})=1$, there exists a right-most infection path $\pi_{[0,V_1]}$ connecting 
$(W_{V_1}, V_1)$ to $\Z_{\le -1} \times \{0\}$. Extend $\pi$ to $[V_1, \tau_1]$ by making it 
constant and equal to $W_{V_1}$ on this time interval. Since $\pi_s \le W_s$ for all
$0 \leq s < \tau_1$, we have $(\tau_1,W_{[0,\tau_1]}) \in \mathcal{R}^{\pi}_{\tau_1} \vee 
\sigma(N_{[0,\tau_1]}, U_{[1,N_{\tau_1}]})$. Note that $\pi$ is not an infection path, 
but only because of a possible healing event at time $\tau_1$, which does not affect 
$(\xi_{\tau_1}(x + W_{V_1}))_{x \le -1}$. 
Therefore, by Lemma~\ref{lemma:stochdomleftpi}, 
the distribution of $\eta_1(\cdot):=\xi_{\tau_1}(\cdot + W_{V_1})$ given $(\tau_1, W_{[0,\tau_1]})$
is stochastically larger 
than $\bar\nu_{\lambda}$. Moreover, 
$(\eta_1(x))_{x < W_{V_1}}$ is independent of $W_{\tau_1}-W_{V_1}$,
and so we need only verify that the distribution of 
$\eta_2:=\xi_{\tau_1}(\cdot+W_{\tau_1}) = \theta_{W_{\tau_1}-W_{V_1}} \eta_1$ is
stochastically larger than $\bar \nu_{\lambda}$ for each possible outcome of $W_{\tau_1}-W_{V_1} \in \{0,\pm 1\}$.
Since, by the definitions of $V_1$ and $\tau_1$, $W_{\tau_1} \neq W_{V_1}$ if and only if $\xi_{\tau_1}(W_{V_1})=1$,
the possible choices for $\eta_2$ are: $\eta_1$ if $W_{\tau_1} = W_{V_1}$,
$\theta_{-1} \eta_1$ if $W_{\tau_1} = W_{V_1} - 1$, or $\theta_1 \eta_1$ if $W_{\tau_1} = W_{V_1}$.
In the latter case, $\eta_2(-1)=1$, and therefore all three possibilities are indeed stochastically larger than
$\bar \nu_{\lambda}$ as claimed.
\end{proof}

By Lemma~\ref{claim:stdomafterV}, there exists a configuration $\xi^{(2)}_0$ distributed as 
$\bar\nu_{\lambda}$, independent of $(\tau_1, W_{[0,\tau_1]})$ and smaller 
than $\xi^{(1)}_{\tau_1}(\cdot + W_{\tau_1})$. We may now define $\xi^{(2)}$ by using 
the events of the graphical representation that lie above time $\tau_1$ with the 
origin shifted to $W_{\tau_1}$, using $\xi^{(2)}_0$ as starting configuration. We may 
then define $W^{(2)}$ and $\tau_2$ from $\xi^{(2)}$, $(N_{t+\tau_1}-N_{\tau_1})_{t \ge 0}$ 
and $(U_k)_{k > N_{\tau_1}}$. With this coupling, clearly $W^{(2)}_t \le W^{(1)}_{\tau_1+t}
-W^{(1)}_{\tau_1}$ for all $t \geq 0$. Furthermore, since $\xi^{(2)}_0$ is independent of 
$(\tau_1,W_{[0,\tau_1]})$, the distribution of $\xi^{(2)}_{\tau_2}(\cdot+W^{(2)}_{\tau_2})$ 
given $(W^{(i)}_{[0, \tau_i]}, \tau_i)_{i=1,2}$ depends only on the random variables with 
$i=2$ and hence, by Lemma~\ref{claim:stdomafterV}, is again stochastically larger than 
$\bar\nu_\lambda$.

We may therefore repeat the argument. More precisely, suppose by induction that we have 
defined $\xi^{(k)}$, $W^{(k)}$ and $\tau_k$ for $k=1,\ldots,n$ and $n \ge 2$, in such a way 
that:
\begin{enumerate}
\item[(B1)] 
$W^{(k+1)}_t \le W^{(k)}_{\tau_n+t} - W^{(n)}_{\tau_n}$ for all $t \ge 0$ and 
$k=1\ldots n-1$;
\item[(B2)] 
$(W^{(k)}, \tau_k)$ is distributed as $\left( W, \tau_1 \right)$ under 
$\mathbb{P}_{\bar \nu_\lambda}$ for all $k=1,\ldots,n$;
\item[(B3)] 
$(W^{(k)}_{[0,\tau_k]}, \tau_k)_{k=1}^{n}$ is i.i.d.;
\item[(B4)] 
The law of $\xi^{(n)}(\cdot + W^{(n)}_{\tau_n})$ given $(W^{(k)}_{[0, \tau_k]},
\tau_k)_{k=1}^n$ is stochastically larger than $\bar \nu_\lambda$.
\end{enumerate}
Then we proceed as before: there exists a configuration $\xi^{(n+1)}_0$ distributed as 
$\bar \nu_\lambda$, smaller than $\xi^{(n)}(\cdot + W^{(n)}_{\tau_n})$ 
and independent of $(W^{(k)}_{[0, \tau_k]},\tau_k )_{k=1}^n$, from which we obtain 
$\xi^{(n+1)}$, $W^{(n+1)}$ and $\tau_{n+1}$, and we prove (B1)--(B4) as in the case 
$n=2$. This settles the existence of the sequence $(W^{(n)}, \tau_n)_{n\in\N}$. All 
that is left to show is that $\hat v(\lambda) > v_0$ and $\lim_{\lambda\to\infty} 
\hat v(\lambda)=v_1$.

Note that Lemma~\ref{lemma:speedvsdensity} is valid also for $\hat W$, and write 
$\hat N^1_t$ to denote the number of jumps that $\hat W$ takes on infected sites. 
Then $\hat N^1_{T_n}$ has distribution $\mathrm{BINOM}(n,\gamma/(1+\gamma))$, and 
by standard arguments we obtain
\begin{equation}
\label{convhatN}
\lim_{t \to \infty} t^{-1}\hat N^1_t 
= \frac{\gamma}{(1+\gamma)\mathbb{E}_{\bar\nu_\lambda}[\tau_1]} > 0,
\end{equation}
which proves $\hat{v}(\lambda) > v_0$. Furthermore, we claim that $\lim_{\lambda\to
\infty} \mathbb{E}_{\bar\nu_\lambda}[V_1]=0$. Indeed, $V_1$ is nonincreasing in 
$\lambda$ and, since $\lim_{\lambda \to \infty}\rho_\lambda = 1$ (recall 
Section~\ref{sec:model}), it is not hard to see that $V_1$ converges in probability 
to zero as $\lambda \to \infty$. Therefore $\lim_{\lambda \to \infty} \mathbb{E}_{
\bar \nu_\lambda}[\tau_1] = 1/(1+\gamma)$, and so $\lim_{\lambda \to \infty} 
\hat{v}(\lambda) = v_1$.
\end{proof}


\section{Regeneration, functional CLT and LDP}
\label{sec:regeneration}

The proof of Theorem~\ref{LLNregen} depends on the construction of \emph{regeneration 
times}, i.e., times at which the random walk forgets its past. This construction 
will be carried out in Section~\ref{subsec:regenconst} and is based on two propositions 
(Propositions~\ref{thm:tauregtime}--\ref{thm:momregtau} below), which are proved in
Sections~\ref{subsec:proofreg1}--\ref{subsec:proofreg2}. At the end of 
Section~\ref{subsec:regenconst} we will see that these propositions imply 
Theorem~\ref{LLNregen}(a,c). The proof of Theorem~\ref{LLNregen}(b) is deferred to
Section~\ref{subsec:speedcont}.


\subsection{Regeneration times}
\label{subsec:regenconst}

If the infection propagation speed $\iota=\iota(\lambda)$ is larger than $|v_0|\vee|v_1|$, 
the maximum absolute speed at which the random walk can move, then each time $W$ finds 
itself on an infected site it can become ``trapped'' forever in an infection cluster 
generated by this site alone. In that case, by Lemma~\ref{depinclusters}, the future 
increments of $W$ become independent of its past. The issue is therefore to find enough 
moments when $W$ sits on an infection. This can be dealt with in a way similar to what 
was done in the proof of $v(\lambda) > v_0$ in Section~\ref{subsec:proofvlarge}. 

\paragraph{Hitting, failure and trial times.}
In order to build the regeneration structure, we first need to extend some definitions 
related to clusters and right-most infections. For $s \ge t$ and $x \in \Z$, let
\begin{equation}
\label{extdefC}
C_{t,s}(x) := \big\{y \in \Z \colon\, (x,t) \leftrightarrow (y,s)\big\}
\end{equation}
and
\begin{equation}
\label{extdefRLx}
R_{t,s}(x) := \sup C_{t,s}(x), \qquad L_{t,s}(x) := \inf C_{t,s}(x).
\end{equation} 
Furthermore, let
\begin{equation}\label{extdefRMinf}
r_{t,s}(x) := \sup_{\substack{y < x \\ \xi_t(y) =1}} R_{t,s}(y),
\end{equation}
i.e., $r_{t,s}(x)$ is the right-most infection at time $s$ that comes from $ \Z_{\le x-1} \times 
\{t\}$.

For $t\ge0$ and $z \in \Z$, let
\begin{equation}
\label{defregV}
V_t(z) : = \inf \big\{ s > t \colon\, W_s = r_{t,s}(z) \big\}
\end{equation}
be the first time after time $t$ at which $W$ meets the right-most infection coming from 
$\Z_{\le z-1} \times \{t\}$. We will call this the \emph{z-wave hitting time} after $t$. It is 
not hard to see that $V_t(z) < \infty$ $\mathbb{P}_{\nu_\lambda}$-a.s.\ for any $t$ and 
$z \le W_t$. Indeed, at any time $t$ there is an infected site $x < z$ whose infection 
survives forever, and in this case $\lim_{s \to \infty}s^{-1}R_{t,s}(x) = \iota > |v_0|
\vee|v_1|$. Therefore there must be an $s>t$ for which $R_{t,s}(x) = W_s$. By 
right-continuity, $\mathbb{P}_{\nu_\lambda}(V_t(z) < \infty \; \forall \; z \le W_t, t 
\ge 0)=1$ as well.

Now define the first \emph{failure time} after time $t$ by (see 
Fig.~\ref{fig:failtrialtimes})
\begin{equation}
\label{defregF}
F_t := \inf\big\{s > t \colon\, W_s \notin [L_{t,s}(W_t), R_{t,s}(W_t)]\big\},
\end{equation}
i.e., the first time after time $t$ when $W$ exits the region surrounded by the cluster of 
$(W_t,t)$. To keep track of the space-time region on which the failure time depends, 
define, for $t \ge 0$ and $x \in \Z$,
\begin{equation}
\label{defregY}
\left(Y_{t,s}(x)\right)_{s \ge t}
\end{equation}
as the process with values in $\Z$ that starts at time $t$ at site $x$ and jumps down by 
following the infection arrows to the left in the graphical representation (see 
Fig.~\ref{LeftPoisson}). Then, given $\mathcal{G}_t$, $(x -Y_{t,t+s}(x))_{s \ge 0}$
is a Poisson process with rate $\lambda$. 

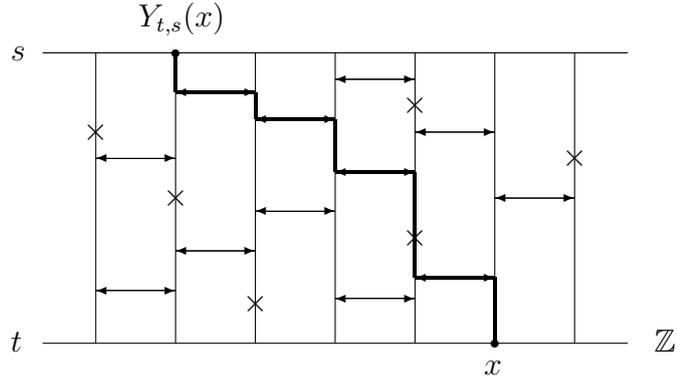
\begin{figure}[hbtp]
\vspace{1cm}
\begin{center}
\setlength{\unitlength}{0.35cm}
\begin{picture}(20,10)(0,0)
\put(0,0){\line(22,0){22}} 
\put(0,11){\line(22,0){22}}
\put(2,0){\line(0,1){11}}
\put(5,0){\line(0,1){11}}
\put(8,0){\line(0,1){11}}
\put(11,0){\line(0,1){11}}
\put(14,0){\line(0,1){11}}
\put(17,0){\line(0,1){11}}
\put(20,0){\line(0,1){11}}
\put(2,2){\vector(1,0){3}} \put(5,2){\vector(-1,0){3}}
\put(2,7){\vector(1,0){3}} \put(5,7){\vector(-1,0){3}}
\put(5,3.5){\vector(1,0){3}} \put(8,3.5){\vector(-1,0){3}}
\put(5,9.5){\vector(1,0){3}} \put(8,9.5){\vector(-1,0){3}}
\put(8,5){\vector(1,0){3}} \put(11,5){\vector(-1,0){3}}
\put(8,8.5){\vector(1,0){3}} \put(11,8.5){\vector(-1,0){3}}
\put(11,1.7){\vector(1,0){3}} \put(14,1.7){\vector(-1,0){3}}
\put(11,6.5){\vector(1,0){3}} \put(14,6.5){\vector(-1,0){3}}
\put(11,10){\vector(1,0){3}} \put(14,10){\vector(-1,0){3}}
\put(14,2.5){\vector(1,0){3}} \put(17,2.5){\vector(-1,0){3}}
\put(14,8){\vector(1,0){3}} \put(17,8){\vector(-1,0){3}}
\put(17,5.5){\vector(1,0){3}} \put(20,5.5){\vector(-1,0){3}}
\put(1.53,7.7){$\times$}
\put(4.53,5.2){$\times$}
\put(7.53,1.2){$\times$}
\put(13.53,8.7){$\times$}
\put(13.53,3.7){$\times$}
\put(19.53,6.7){$\times$}
{\linethickness{0.04cm}
\put(17,0){\line(0,1){2.5}}
\put(17,2.5){\line(-1,0){3}}
\put(14,2.5){\line(0,1){4}}
\put(14,6.5){\line(-1,0){3}}
\put(11,6.5){\line(0,1){2}}
\put(11,8.5){\line(-1,0){3}}
\put(8,8.5){\line(0,1){1}}
\put(8,9.5){\line(-1,0){3}}
\put(5,9.5){\line(0,1){1.5}}
}
\put(-1.2,-.3){$t$}
\put(-1.2,10.7){$s$}
\put(23,-0.3){$\mathbb{Z}$}
\put(16.6,-1.2){$x$}
\put(3.6,12){$Y_{t,s}(x)$}
\put(17,0){\circle*{.35}}
\put(5,11){\circle*{.35}}
\end{picture}
\end{center}
\caption{\small $Y_{t,s}(x)$ starts at $x$ and goes upwards and to the left across 
the arrows of the graphical representation.} 
\label{LeftPoisson}
\end{figure}

With the above observations we can define the \emph{trial time} after a failure 
time (see Fig.~\ref{fig:failtrialtimes}):
\begin{equation}
\label{defregT}
T_t := \left\{
\begin{array}{ll}
\infty & \text{if } F_t = \infty, \\
V_{F_t}(Y_{t,F_t}(W_t)) & \text{otherwise,}
\end{array}\right.
\end{equation}
i.e., $T_t$ is the $Y_{t,F_t}(W_t)$-wave time after time $F_t$ when the latter is 
finite. This wave ensures ``good conditions'' at the trial time, meaning an appreciable 
density of infections to the left of $W$.

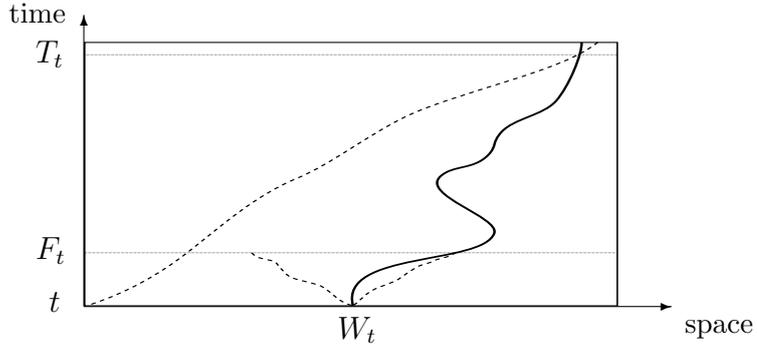
\begin{figure}[htb]
\vspace{0.75cm}
\begin{center}
\begin{picture}(250,100)
\put(30,0){\includegraphics[height=100pt,width=200pt]{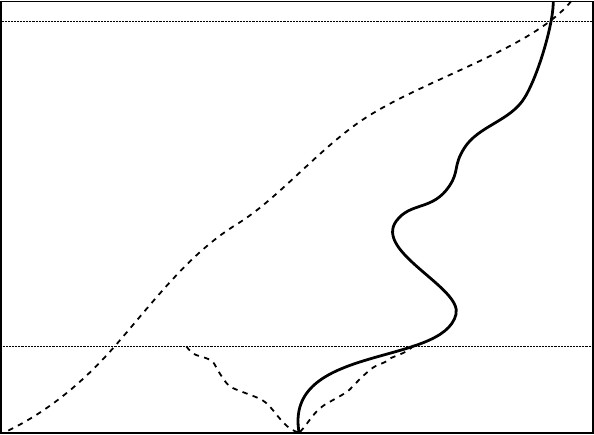}}
\put(30,0){\vector(1,0){220}}
\put(30,0){\vector(0,1){110}}
\put(2,107){\small\mbox{time}}
\put(255,-10){\small\mbox{space}}
\put(125,-12){$W_t$}
\put(12,92){$T_t$}
\put(12,18){$F_t$}
\put(17,-2){$t$}
\vspace{0.5cm}
\end{picture}
\end{center}
\caption{\small A failure time $F_t$ and a trial time $T_t$ after time $t$. The dashed 
lines represent infection paths. The thick line represents the path of $W$.} 
\label{fig:failtrialtimes}
\end{figure}

\paragraph{Regeneration times.}
We can now define our regeneration time $\tau$. First let
\begin{equation}
\label{defregcurlyT1}
\mathcal{T}_1 := V_0(0)
\end{equation}
and, under the assumption that $\mathcal{T}_1,\ldots,\mathcal{T}_k$, $k \in \N$, are 
all defined, let
\begin{equation}
\label{defregcurlyTk}
\mathcal{T}_{k+1} := \left\{
\begin{array}{ll}
\infty & \text{if } \mathcal{T}_k = \infty, \\
T_{\mathcal{T}_k} & \text{otherwise.}
\end{array}\right.
\end{equation}
Note that the $\mathcal{T}_k$'s are stopping times w.r.t.\ the filtration $\mathcal{G}$. 
Finally, put
\begin{equation}
\label{defregK}
K := \inf \big\{k \in \N \colon\, \mathcal{T}_k < \infty, \mathcal{T}_{k+1}=\infty\big\},
\end{equation}
and let
\begin{equation}
\label{defregtau}
\tau := \mathcal{T}_K.
\end{equation}
Note that $K < \infty$ a.s.\ since, at any trial time, the probability for the next 
failure time to be infinite is uniformly bounded from below. We will prove in
Sections~\ref{subsec:proofreg1}--\ref{subsec:proofreg2} that $\tau$ is a regeneration 
time and has exponential moments. This is stated in the following two propositions. 

\begin{proposition}
\label{thm:tauregtime}
The distribution of $(W_{t+\tau}-W_\tau)_{t \ge 0}$ under both $\mathbb{P}_{\nu_\lambda}
(\cdot \mid \tau, W_{[0,\tau]})$ and $\mathbb{P}_{\nu_\lambda}(\cdot \mid \Gamma, \tau, 
W_{[0,\tau]})$ is the same as that of $W$ under $\mathbb{P}_{\nu_\lambda}(\cdot \mid 
\Gamma)$, where
\begin{equation}
\label{defGamma} 
\Gamma := \{\xi_0(0)=1, F_0 = \infty \}.
\end{equation}
\end{proposition}

\begin{proposition}
\label{thm:momregtau}
$\tau$ and $|W_\tau|$ have exponential moments under both $\mathbb{P}_{\nu_\lambda}$ 
and $\mathbb{P}_{\nu_\lambda}(\cdot \mid \Gamma)$, uniformly in $\lambda \in 
[\lambda_-,\lambda_+]$ for any fixed $\lambda_-, \lambda_+ \in (\lambda_W, \infty)$.
\end{proposition}

These two propositions imply the LLN and Theorem~\ref{LLNregen}(a), with
\begin{equation}
\label{formulasreg}
v(\lambda) = \frac{\mathbb{E}_{\nu_{\lambda}}
\left[W_{\tau}\mid \Gamma\right]}{\mathbb{E}_{\nu_{\lambda}}
\left[\tau\mid \Gamma\right]},
\qquad \sigma(\lambda)^2 = \frac{\mathbb{E}_{\nu_{\lambda}}
\left[(W_{\tau})^2\mid \Gamma\right] - \mathbb{E}_{\nu_{\lambda}}
\left[W_{\tau}\mid \Gamma\right]^2}{\mathbb{E}_{\nu_{\lambda}}
\left[\tau\mid \Gamma\right]}.
\end{equation}
They also imply that
\begin{equation}\label{expdecay}
\limsup_{t \to \infty} \frac{1}{t} \log \mathbb{P}_{\nu_\lambda}
\left(t^{-1}W_t \notin (v-\epsilon, v+ \epsilon) \right) < 0 
\quad \forall \; \epsilon > 0.
\end{equation}
For a proof of these facts, the reader can follow word-by-word the arguments 
given in Avena, dos Santos and V\"ollering~\cite{AvdSaVo12}, Theorem 3.8 and 
Section 4.1 (which do not require \eqref{condjumprates1}--\eqref{condjumprates2}).

Theorem~\ref{LLNregen}(c) follows from \eqref{expdecay} and the partial LDP 
proven in Avena, den Hollander and Redig \cite{AvdHoRe10} for attractive 
spin-flip systems (including the contact process). Here, partial means that the 
LDP is shown to hold outside a possible interval where the rate function is zero.
However, \eqref{expdecay} precisely precludes the presence of such an interval. 
(See Glynn and Whitt~\cite{GlWh94}, Theorem 3, for more details.)

The proof of Theorem~\ref{LLNregen}(b) is deferred to Section~\ref{subsec:speedcont}.

\begin{remark}
\label{rem:extensionsCLTLDP}
\emph{
It is easy to check that $\tau < \infty$ $\mathbb{P}_{\eta}$-a.s.\ whenever $\eta$ contains 
infinitely many infections. Hence the FCLT also holds under $\mathbb{P}_{\eta}$ for any 
such $\eta$. Moreover, the result in \cite{AvdHoRe10} concerning the partial LDP holds
whenever the initial measure $\mu$ has positive correlations. If additionally $\mu$ is 
stochastically larger than a non-trivial Bernoulli product measure, then it is possible to 
show that $\tau$ has exponential moments under $\mathbb{P}_{\mu}$, which implies that 
the LDP also holds under $\mathbb{P}_{\mu}$, with a possibly different rate function 
that however still has a unique zero at $v$. A proof that $\tau$ has exponential moments
in this situation does not follow directly from the proof of Theorem~\ref{thm:momregtau} given below, but 
can be obtained with the help of the method that was used in the proof of 
Lemma~\ref{lemma:momZdelta}.
}
\end{remark}


\subsection{Proof of Proposition~\ref{thm:tauregtime}}
\label{subsec:proofreg1}

We first show that the regeneration strategy indeed makes sense.

\begin{lemma}
\label{lemma:distreg}
For all $t \ge 0$,
\begin{equation}
\label{eq:distreg}
\mathbb{P}_{\nu_\lambda}\left(F_t = \infty, \left(W_{s+t}-W_t\right)_{s \ge 0} 
\in \cdot \mid \mathcal{G}_t \right) 
= \mathbb{P}_{\mathbbm{1}_0}\left(\Gamma_0, W \in \cdot~ \right) 
\text{ a.s. on } \{ \xi_t(W_t) =1 \},
\end{equation}
where $\Gamma_0 := \{F_0 = \infty\}$. The same is true for a finite stopping time w.r.t.\
$\mathcal{G}$ replacing $t$.
\end{lemma}

\begin{proof}
First note that $\mathbb{P}_{\eta}(\Gamma_0, W \in \cdot) = \mathbb{P}_{\mathbbm{1}_0}
(\Gamma_0, W \in \cdot)$ for any $\eta$ with $\eta(0)=1$. This follows from 
Lemma~\ref{depinclusters} because, on $\Gamma_0$, $W$ depends on $\xi$ only through 
$\{\xi_t(x) \colon\, t \ge 0, x \in [L_t(0), R_t(0)]\}$, and $\Gamma_0$ does not depend 
on $\xi_0$. Now, letting $\hat \xi_t(\cdot) := \xi_t(\cdot + W_t)$, we can write
(recall \eqref{defregF})
\begin{multline}
\label{pdistregeq1}
\mathbb{P}_{\nu_\lambda} \Big(\xi_t(W_t) = 1, F_t = \infty, 
(W_{s+t}-W_t)_{s \ge 0} \in \cdot \, \Big| \, \mathcal{G}_t \Big) \\ 
= \mathbb{E}_{\nu_\lambda}\left[\xi_t(W_t) \mathbb{P}_{\hat \xi_t}(\Gamma_0, W \in \cdot~) 
\, \middle| \, \mathcal{G}_t \right] = \xi_t(W_t)\, \mathbb{P}_{\mathbbm{1}_0}(\Gamma_0, W \in \cdot~),
\end{multline}
where the first equality is justified by the Markov property and the translation invariance 
of the graphical representation. To extend the result to stopping times we can use the 
strong Markov property of $(\xi,W)$.
\end{proof}

With the help of Lemma~\ref{lemma:distreg} we are ready to prove 
Proposition~\ref{thm:tauregtime}.

\begin{proof}
We will closely follow the proof of Theorem 3.4 in \cite{AvdSaVo12}. Let $\mathcal{G}_{\tau}$ 
be the $\sigma$-algebra of all events $B$ such that, for all $n \in \N_0$, there exists 
a $B_n \in \mathcal{G}_{\mathcal{T}_n}$ such that $B \cap \{K=n\} = B_n \cap \{K=n\}$. 
Note that $\tau$ and $W_{[0,\tau]}$ are in $\mathcal{G}_{\tau}$.

In the following, we abreviate $W^{(t)} := \left(W_{s+t}-W_t\right)_{s \ge 0}$. Pick $f$ 
bounded and measurable, $B \in \mathcal{G}_{\tau}$, and write (recall \eqref{defregcurlyTk})
\begin{equation}
\label{exprels}
\begin{aligned}
&\E_{\nu_{\lambda}}\left[\mathbbm{1}_{B} f(W^{(\tau)})\right] 
= \sum_{n\in\N_0} \E_{\nu_{\lambda}}\left[\mathbbm{1}_{B_n} 
\mathbbm{1}_{\{K=n\}} f(W^{(\mathcal{T}_{n})}) \right] \\
&= \sum_{n\in\N_0} \E_{\nu_{\lambda}}\Big[\mathbbm{1}_{B_n} 
\mathbbm{1}_{\{\mathcal{T}_{n} < \infty\}}\, 
\E_{\nu_{\lambda}}\left[\mathbbm{1}_{\{F_{\mathcal{T}_{n}} = \infty\}} 
f(W^{(\mathcal{T}_{n})}) \,\middle|\, \mathcal{G}_{\mathcal{T}_{n}}\right] \Big].
\end{aligned}
\end{equation}
Since $\xi_{\mathcal{T}_n}(W_{\mathcal{T}_n})=1$ on $\{\mathcal{T}_n < \infty\}$, by 
Lemma~\ref{lemma:distreg} the last line of \eqref{exprels} equals
\begin{equation}
\begin{aligned}
&\E_{\mathbbm{1}_0}\left[ f(W) \mathbbm{1}_{\Gamma_0} \right] 
\sum_{n\in\N_0} \E_{\nu_{\lambda}}\left[\mathbbm{1}_{B_n} 
\mathbbm{1}_{\{\mathcal{T}_{n}<\infty\}} \right] \\
 &= \E_{\mathbbm{1}_0}\left[ f(W) \mid \Gamma_0 \right] 
\sum_{n\in\N_0} \E_{\nu_{\lambda}}\left[\mathbbm{1}_{B_n} 
\mathbbm{1}_{\{\mathcal{T}_{n}<\infty\}} \right] \P_{\mathbbm{1}_0}(\Gamma_0),
\end{aligned}
\end{equation}
which, again by Lemma~\ref{lemma:distreg}, equals
\begin{equation}
\label{compproptau}
\begin{aligned}
&\E_{\mathbbm{1}_0}\left[ f(W) \mid \Gamma_0 \right] 
\sum_{n\in\N_0} \E_{\nu_{\lambda}}\left[\mathbbm{1}_{B_n} 
\mathbbm{1}_{\{\mathcal{T}_{n}<\infty\}} \P_{\nu_{\lambda}}
\left(F_{\mathcal{T}_{n}} = \infty \,\middle|\, 
\mathcal{G}_{\mathcal{T}_{n}}\right) \right]\\
&= \E_{\mathbbm{1}_0}\left[ f(W) \mid \Gamma_0 \right] 
\sum_{n\in\N_0} \P_{\nu_{\lambda}}\left(B_n, K=n \right)\\
&= \E_{\mathbbm{1}_0}\left[ f(W) \mid \Gamma_0 \right] \P_{\nu_{\lambda}}(B)\\[0.2cm] 
&= \E_{\nu_{\lambda}}\left[ f(W) \mid \Gamma \right] \P_{\nu_{\lambda}}(B),
\end{aligned}
\end{equation}
where the last equality is, one more time, justified by Lemma~\ref{lemma:distreg}. This 
proves the claim under $\mathbb{P}_{\nu_\lambda}$. 

To extend the claim to $\mathbb{P}_{\nu_\lambda}(\cdot \mid \Gamma)$, note that 
$\Gamma \in \mathcal{G}_{\tau}$ since
\begin{equation}\label{gammainGtau}
\Gamma \cap \{K=n\} = \big\{\xi_0(0)=1, W_s \in [L_s(0), R_s(0)] \; 
\forall \; s \in [0, \mathcal{T}_n] \big\} \cap \{K=n\},
\end{equation}
and apply \eqref{compproptau} to $B \cap \Gamma$ instead of $B$.
\end{proof}


\subsection{Proof of Proposition~\ref{thm:momregtau}}
\label{subsec:proofreg2}

\paragraph{Exponential moments.}
We first show that $T_0$ has exponential moments when it is finite, uniformly 
for $\lambda$ in compact sets. Fix $\lambda_-, \lambda_+\in(\lambda_W,\infty)$,
$\lambda_- \le \lambda_+$.

\begin{lemma}
\label{lemma:momT0}
For every $\lambda \in [\lambda_-,\lambda_+]$ and $\epsilon > 0$ there exists an $a 
= a(\lambda_-,\lambda_+,\epsilon) > 0$ such that, for any probability measure $\mu$ 
stochastically larger than $\bar \nu_\lambda$,
\begin{equation}
\label{eqmomT0}
\begin{array}{rl}
\text{(a)} & \mathbb{E}_{\mu}\left[\mathbbm{1}_{\{T_0 < \infty \}} 
e^{a T_0}\right] \le 1+ \epsilon.\\
\text{(b)} & \mathbb{E}_{\mu}\left[e^{a V_0(0)}\right] \le 1+ \epsilon.
\end{array}
\end{equation}
\end{lemma}

\begin{proof}
We couple systems with infection rates $\lambda_-$, $\lambda$ and $\lambda_+$ starting, 
respectively, from $\bar \nu_{\lambda_-}$, $\mu$ and $\mathbf{1}$, by coupling their 
initial configurations and their infection events monotonically. Denote their joint 
law by $\mathbb{P}$. In what follows, we will refer to these systems by their rates 
and we will use a superscript to indicate on which system a random variable depends.

We will bound $T_0 \mathbbm{1}_{\{T_0 < \infty\}} = T_0 \mathbbm{1}_{\{F_0 < \infty\}}$ 
by a time $D_0$ that depends only on systems $\lambda_{\pm}$ and has exponential moments 
under $\mathbb{P}$. We start by bounding $F_0 \mathbbm{1}_{\{F_0 < \infty\}}$ by a 
variable $D_1$ depending only on system $\lambda_-$.
Let
\begin{equation}
\label{momT0eq1}
r_t := \sup_{x \in \Z_{\le 0}}R_t(x), \qquad l_t := \inf_{x \in \Z_{\ge 0}}L_t(x).
\end{equation}
Then $r_t$ is the same as $r_{0,t}(0)$ in \eqref{extdefRMinf} when all sites in 
$\Z_{\le 0}$ are infected, and analogously for $l_t$. Furthermore, $R_t(0), L_t(0)$ are 
equal to $r_t,l_t$ while $C_t(0)\neq\emptyset$: this can be seen by using the graphical
representation (see e.g.\ Liggett~\cite{Li85} Chapter VI, Theorem 2.2). Therefore
\begin{equation}
\label{momT0eq2}
F_0 = \inf\{ t \ge 0 \colon\, r_t < W_t \text{ or } l_t > W_t\}.
\end{equation}
Let $m:= \frac12(\iota(\lambda_-)+|v_0| \vee |v_1|)$. Take homogeneous random walks 
$X^{i}$ jumping at rates $\alpha_i$,$\beta_i$, $i\in\{0,1\}$, independent of $\xi$ and 
coupled to $W$ in such a way that $X^0_t \le W_t \le X^1_t$ for all $t \ge 0$. Set
\begin{equation}
\label{momT0eq3}
\begin{array}{rcl}
D_{1a} & := & \sup\big\{t \ge 0 \colon\, l^{\lambda_-}_t \ge -mt 
\text{ or } r^{\lambda_-}_t \le mt \big\},\\
D_{1b} & := & \sup\big\{t \ge 0 \colon\, |X^0_t| \vee |X^1_t| > mt\big\}.
\end{array}
\end{equation}
Then $D_{1a}$ depends only on system $\lambda_-$ and has exponential moments by known 
large deviation bounds for $r_t$ (see Liggett~\cite{Li85} Chapter VI, Corollary 3.22), 
while $D_{1b}$ is independent of $\xi$ and has exponential moments by standard large 
deviation bounds for $X^0$ and $X^1$. Noting that $r_t$ and $l_t$ are monotone, we can 
take $D_1 := D_{1a} \vee D_{1b}$, which does not depend on the initial configuration.

Set $\delta := \frac12(\iota(\lambda_-) - m)$, $x_0 := Y^{\lambda_+}_{0,D_1}(0)
-\lceil(\iota(\lambda_+)+\delta)D_1 \rceil$ and note, using the graphical representation, 
that $\Delta_0 := x_0 - Z^{\lambda_-}_{\delta}(x_0)$ is independent of $x_0$, where 
$Z_{\delta}(x)$ is as in \eqref{defZdelta}. Then
\begin{equation}
\label{momT0eq4}
D_0 := \frac{\Delta_0 + |x_0| + 1}{\iota(\lambda_-) - \delta - m} = 4 \frac{\Delta_0 + |x_0| + 1}{\iota(\lambda_-) - |v_0|\vee|v_1|}
\end{equation}
depends only on $\lambda_-$, $\lambda_+$ and has exponential moments under $\mathbb{P}$ by 
Lemma~\ref{lemma:momZdelta}. It is easy to check that $D_0$ is the intersection time of the 
line of inclination $\iota(\lambda_-) -\delta$ passing through $(Z^{\lambda_-}_{\delta}(x_0)
-1,0)$ and the line of inclination $m$ passing through the origin. Since the system $\lambda$ 
has more infections than the system $\lambda_-$ and $D_0 \ge D_1$, we have $T_0 \mathbbm{1}_{
\{T_0 < \infty\}} \le D_0$, which proves (a). For (b), we can bound $V_0(0)$ analogously, 
taking $x_0 = 0$ instead.
\end{proof}

\paragraph{Infections at trial times.}
We next show that at trial times there are more infections to the left of the random walk
than under $\bar \nu_\lambda$.

\begin{lemma}
\label{lemma:densitytrial}
For all $n \in \N$, on the event $\{\mathcal{T}_n < \infty\}$ the law of $\xi_{\mathcal{T}_n}
(~\cdot+W_{\mathcal{T}_n})$ under $\mathbb{P}_{\nu_\lambda}(\cdot \mid \mathcal{T}_{[1,n]}, 
W_{[0,\mathcal{T}_n]})$ a.s.\ is stochastically larger than $\bar \nu_\lambda$.
\end{lemma}

\begin{proof}
Suppose that $n\ge 2$ (the case $n=1$ is simpler). Using the definition of $\mathcal{T}_n$, 
we can show by induction that, if $\mathcal{T}_n<\infty$, then there exist infection paths 
connecting $(W_{\mathcal{T}_n},\mathcal{T}_n)$ to $\Z_{\le -1} \times \{0\}$ and never 
touching the paths $Y^{\mathcal{T}_{k}}(W_{\mathcal{T}_{k}})$, $k=1,\ldots,n-1$, or the 
region to the right of $W$. Take $\pi$ to be the right-most of these infection paths. Then 
$\pi$ is a random infection path with properties (p1) and (p2), and 
\begin{equation}
(\mathcal{T}_{[1,n]},W_{[0,\mathcal{T}_n]}) \in \mathcal{R}^{\pi}_{\mathcal{T}_n} 
\vee \sigma\big(N_{[0,\mathcal{T}_n]}, U_{[1, N_{\mathcal{T}_n}]}\big).
\end{equation} 
Therefore the result follows from Lemma~\ref{lemma:stochdomleftpi}.
\end{proof}

\paragraph{Conclusion.}
We are now ready to prove Proposition~\ref{thm:momregtau}.

\begin{proof}
Let 
\begin{equation}\label{defkappa}
\kappa := \mathbb{P}_{\mathbbm{1}_0}(\Gamma_0).
\end{equation}
By Lemma \ref{lemma:distreg}, $\mathbb{P}_{\nu_\lambda}(\Gamma) = \kappa \rho_\lambda 
\ge \kappa \rho_{\lambda_-}$ by monotonicity (recall the definition of $\rho_\lambda$
from Section~\ref{sec:model}).  Also, there exists a $\kappa_- > 0$ such that $\kappa 
\ge \kappa_-$ for any $\lambda \ge \lambda_-$: we can take $\kappa_-$ to be the probability 
that $X^0$ and $X^1$ in the proof of Lemma~\ref{lemma:momT0} never cross $L(0)$ or 
$R(0)$ in system $\lambda_-$. Therefore it is enough to prove the claim for $\mathbb{P}_{
\nu_\lambda}$. Since $|W_t|$ is dominated by $N_t$, which is a Poisson process independent of 
$\xi$, we only need to worry about $\tau$.

For $\epsilon > 0$ such that $(1+\epsilon)(1-\kappa_-)< 1$, take $a > 0$ as in 
Lemma~\ref{lemma:momT0}. On the event $\{\mathcal{T}_n < \infty\}$, let $\hat \xi_n 
:= \xi_{\mathcal{T}_n}(~\cdot + W_{\mathcal{T}_n})$ and note that, given 
$\mathcal{G}_{\mathcal{T}_n}$, $\mathcal{T}_{n+1}-\mathcal{T}_n$ is distributed as $T_0$ 
under $\mathbb{P}_{\hat \xi_n}$. By Lemma~\ref{lemma:densitytrial}, the law of $\hat \xi_n$ 
under $\mathbb{P}_{\nu_\lambda}(~\cdot \mid \mathcal{T}_{[1,n]},W_{[0, \mathcal{T}_n]})$ 
is stochastically larger than $\bar \nu_\lambda$, and we get from Lemma~\ref{lemma:momT0}
that
\begin{equation}
\begin{aligned}
\label{pmomregtaueq1}
&\mathbb{E}_{\nu_\lambda}\Big[\mathbbm{1}_{\{\mathcal{T}_{n+1}<\infty\}}
e^{a (\mathcal{T}_{n+1}-\mathcal{T}_n)} \, \Big| \,
\mathcal{T}_{[1,n]}, W_{[0, \mathcal{T}_n]}\Big] \\
&\qquad = \mathbb{E}_{\nu_\lambda}\Big[\mathbb{E}_{\hat \xi_n}
\left[\mathbbm{1}_{\{T_0<\infty\}} e^{a T_0}\right] 
\, \Big| \, \mathcal{T}_{[1,n]}, W_{[0, \mathcal{T}_n]}\Big] \le 1+\epsilon.
\end{aligned}
\end{equation}
Using this bound, estimate
\begin{equation}
\begin{aligned}
\mathbb{E}_{\nu_\lambda}\left[\mathbbm{1}_{\{\mathcal{T}_{n+1}<\infty\}}
e^{a \mathcal{T}_{n+1}}\right] 
= \; &
\mathbb{E}_{\nu_\lambda}\Big[\mathbbm{1}_{\{\mathcal{T}_{n}<\infty\}}
e^{a \mathcal{T}_{n}} \mathbb{E}_{\nu_\lambda}\left[\mathbbm{1}_{\{\mathcal{T}_{n+1}<\infty\}}
e^{a (\mathcal{T}_{n+1}-\mathcal{T}_n)} \, \big| \, \mathcal{T}_n\right]\Big] \\
\le \; & (1+\epsilon)\mathbb{E}_{\nu_\lambda}\left[\mathbbm{1}_{\{\mathcal{T}_n<\infty\}}
e^{a \mathcal{T}_n}\right],
\end{aligned}
\end{equation}
so that, by induction,
\begin{equation}
\label{pmomregtaueq3}
\mathbb{E}_{\nu_\lambda}\left[\mathbbm{1}_{\{\mathcal{T}_n<\infty\}}
e^{a \mathcal{T}_n}\right] \le (1+\epsilon)^n.
\end{equation}
Using Lemma~\ref{lemma:distreg}, write, for $n \in \N$,
\begin{equation}\label{Kgeom}
\mathbb{P}_{\nu_\lambda}\left( K \ge n+1\right) 
= \mathbb{P}_{\nu_\lambda}\left( \mathcal{T}_{n}<\infty, 
F_{\mathcal{T}_n} < \infty \right) 
= (1-\kappa)\mathbb{P}_{\nu_\lambda}\left( K \ge n \right)
\end{equation}
to note that $K$ has distribution GEO($\kappa$). To conclude, use 
\eqref{pmomregtaueq3}--\eqref{Kgeom} to write
\begin{equation}
\label{pmomregtaueq4}
\begin{array}{ll}
\vspace{0.2cm}
\E_{\nu_\lambda}\left[e^{\frac{a}{2}\tau} \right] 
& = \sum_{n\in\N} 
\E_{\nu_\lambda}\left[\mathbbm{1}_{\{K=n\}}e^{\frac{a}{2}\mathcal{T}_n} \right] 
= \sum_{n\in\N} \E_{\nu_\lambda}\left[\mathbbm{1}_{\{K=n\}}
\mathbbm{1}_{\{\mathcal{T}_n < \infty\}}e^{\frac{a}{2}\mathcal{T}_n} \right] \\
\vspace{0.2cm}
& \le \sum_{n\in\N} \P_{\nu_\lambda}\left(K=n\right)^{\frac{1}{2}}
\E_{\nu_\lambda}\left[\mathbbm{1}_{\{\mathcal{T}_n< \infty\}}
e^{a \mathcal{T}_n} \right]^{\frac{1}{2}} \\
& \le (1-\kappa_-)^{-\frac{1}{2}}\sum_{n\in\N} 
\left(\sqrt{(1-\kappa_-)(1+\epsilon)}\right)^n < \infty,
\end{array}
\end{equation} 
where in the second line we use the Cauchy-Schwarz inequality.
\end{proof}


\subsection{Continuity of the speed and the volatility}
\label{subsec:speedcont}

Given $\lambda_- \le \lambda_+ $ in $(\lambda_W, \infty)$ and $(\lambda_n)_{n \in \N}$, 
$\lambda_* \in [\lambda_-, \lambda_+]$ such that either $\lambda_n \uparrow \lambda_*$  
or $\lambda_n \downarrow \lambda_*$ as $n\to\infty$, we can simultaneously construct 
systems with infection rates $(\lambda_n)_{n \in \N}$, $\lambda_*$ and $\lambda_{\pm}$, 
starting from equilibrium, with a single graphical representation in the standard fashion, 
taking a monotone sequence of Poisson processes for infection events and coupling the 
initial configurations monotonically. For $n \in \N \cup \{*,+,-\}$, denote by $\Lambda^n 
:= (\xi^n_0,H,I^n,N,U)$ the system with infection rate $\lambda_n$, and by $\mathbb{P}$ 
their joint law. In the following, we will use a superscript $n$ to indicate functionals 
of $\Lambda^n$.

In view of \eqref{formulasreg} and Proposition~\ref{thm:momregtau}, in order to prove 
convergence of $v(\lambda_n)$ and $\sigma(\lambda_n)$ it is enough to prove convergence 
in distribution of $\Gamma^n$ and of $(W^n_{\tau^n},\tau^n)\mathbbm{1}_{\Gamma^n}$. 
The main step to achieve this will be to approximate relevant random variables with 
uniformly large probability by random variables depending on bounded regions of the 
graphical representation.

Note that, by monotonicity and continuity of $\lambda\mapsto\rho_\lambda$ (see 
Liggett~\cite{Li85} Chapter VI, Theorem 1.6),
\begin{equation}
\label{convxi0}
\lim_{n \to \infty} \xi^n_0(x) = \xi^*_0(x) \quad 
\forall \; x \in \Z \quad \mathbb{P} \text{-a.s.}
\end{equation}
Recall the definitions of $F_0$, $\mathcal{T}_k$ and $K$ in \eqref{defregF}, 
\eqref{defregcurlyT1}--\eqref{defregcurlyTk} and \eqref{defregK}, respectively.
For $n \in \N \cup \{*\}$ and $k \in \N$, let
\begin{equation}
\label{defGammak}
\Gamma^{n}_k := \Big\{\xi^{n}_0(0)=1, W^{n}_s \in [L^{n}_s(0), R^{n}_s(0)]  \; 
\forall \; s \in [0,\mathcal{T}^{n}_k] \cap \R \Big\},
\end{equation}
so that $\Gamma^{n} = \Gamma^{n}_k$ on $\{K^{n} = k\}$ as in \eqref{gammainGtau}.

\begin{proposition}
\label{contprop1}
For every $k \in \N$, $(W^n_{\mathcal{T}^n_k},\mathcal{T}^n_k, \mathbbm{1}_{\Gamma^n_k})
\mathbbm{1}_{\{\mathcal{T}^n_k<\infty\}}$, $\mathbbm{1}_{\{\mathcal{T}^n_k < \infty\}}$ 
and $\mathbbm{1}_{\{F^n_0<\infty\}}$ converge in probability as $n \to \infty$ to the 
corresponding functionals of $\Lambda^*$.
\end{proposition}

\begin{proof}
We first show that, for every fixed $T \in (0,\infty)$,
\begin{equation}
\label{contprop1eq1}
(W^n_{\mathcal{T}^n_k},\mathcal{T}^n_k,\mathbbm{1}_{\Gamma^n_k})
\mathbbm{1}_{\{\mathcal{T}^n_k \le T \}},
\quad 
\mathbbm{1}_{\{\mathcal{T}^n_k \le T\}},
\quad 
\mathbbm{1}_{\{F^n_0\le T\}},
\end{equation}
converge a.s.\ as $n \to \infty$ to the corresponding functionals of $\Lambda^*$. To that
end, let $\bar Y_{t,s}(x)$ be the increasing analogue of $Y_{t,s}(x)$ in \eqref{defregY}, 
starting from $x$ but jumping across the arrows of $I$ to the right. Let $\bar{Z}_{\delta}
(x)$, analogously to $Z_{\delta}(x)$ in \eqref{defZdelta}, be the first infected site to 
the right of $x$ whose infection spreads inside a wedge between lines of inclination 
$-(\iota+\delta)$ and $-(\iota-\delta)$. Take $\delta := \iota(\lambda_-)/2$, set 
$\underline{y} := Y^+_{0,T}(-N_T)$ and $\underline{z} := Z^-_{\delta}(\underline{y}
- \lceil (\iota(\lambda_-) +\delta) T \rceil)$. Analogously, put $\overline{y} 
:= \bar Y^+_{0,T}(N_T)$ and $\overline{z} := \bar Z^-_{\delta}(\overline{y} 
+ \lceil (\iota(\lambda_-) +\delta) T \rceil)$.

Now observe that, for any $n \in \N \cup \{*\}$, all random variables in \eqref{contprop1eq1} 
depend on $\Lambda^n$ only in the space-time box $\mathcal{B}:=[\underline{z},\overline{z}]
\times [0,T]$. Indeed, for any $0 \le t \le s \le T$, we have $L^n_{t,s}(W^n_t) \ge Y^n_{t,s}
(W^n_t) \ge y^-$ and $R^n_{t,s}(W^n_t) \le y^+$, so that $\{F^n_t \le s\}$ depends on 
$\Lambda^n$ only inside $[\underline{y},\overline{y}] \times [0,T]$. Also, there are infection 
paths from time $0$ to time $T$ inside $[\underline{z}, \underline{y})$ and $(\overline{y},
\overline{z}]$. Therefore the states of $\xi^n$ inside $[\underline{y},\overline{y}] \times 
[0,T]$ depend on $\Lambda^n$ only in $\mathcal{B}$ (see the proof of Lemma~\ref{depinclusters}). 
The same is true for $\{T^n_t \le s \}$, since any infection path needed to discover $T^n_t$ 
can be taken inside $\mathcal{B}$. Therefore, by \eqref{convxi0} (and since the graphical 
representation is a.s.\ eventually constant inside bounded space-time regions), the claim 
after \eqref{contprop1eq1} follows.

To conclude note that, because $\mathcal{T}_k \mathbbm{1}_{\{\mathcal{T}_k < \infty\}} \le 
\tau$ and $F_0\mathbbm{1}_{\{F_0 < \infty\}} \le T_0 \mathbbm{1}_{\{T_0 < \infty\}}$,
\begin{equation}
\label{contlem1eq1}
\lim_{T \to \infty} \sup_{n \in \N \cup \{*\}} 
\mathbb{P}\big(T < \mathcal{T}^n_k < \infty \; 
\text{ or } \; T < F^n_0 < \infty\big) = 0
\end{equation}
by Proposition~\ref{thm:momregtau} and Lemma~\ref{lemma:momT0}, which implies that, for large 
$T$, the random variables in the statement are equal to the ones in \eqref{contprop1eq1} with uniformly large probability.
\end{proof}

\begin{corollary}
\label{contcor1}
Let $\kappa^n$ be as in \eqref{defkappa}. Then $\lim_{n\to\infty}\kappa^n = \kappa^*$ and 
$K^n$ converges in distribution to $K^*$.
\end{corollary}

\begin{proof}
This follows directly from Proposition~\ref{contprop1} and the definition of $\kappa$ since, 
by \eqref{Kgeom}, $K^n$ is a geometric random variable with parameter $\kappa^n$.
\end{proof}

With these results we can conclude the proof of Theorem~\ref{LLNregen}(c).

\begin{proof}
Let $f$ be a bounded measurable function. For $k \in \N$, write
\begin{equation}
\begin{aligned}
\label{pconteq1}
\mathbb{E}\left[f(W^n_{\tau^n},\tau^n) \mathbbm{1}_{\Gamma^n} 
\mathbbm{1}_{\{K^n = k \}} \right]
&= \mathbb{E}\left[f\big(W^n_{\mathcal{T}^n_k},\mathcal{T}^n_k\big) 
\mathbbm{1}_{\Gamma^n_k} \mathbbm{1}_{\big\{\mathcal{T}^n_k < \infty, 
F_{\mathcal{T}^n_k} = \infty \big\}} \right] \\
&= \kappa^n\,\mathbb{E}\left[f\big(W^n_{\mathcal{T}^n_k},
\mathcal{T}^n_k\big) \mathbbm{1}_{\Gamma^n_k} 
\mathbbm{1}_{\left\{\mathcal{T}^n_k < \infty \right\}} \right]\\
&\xrightarrow[]{n \to \infty} 
\kappa^* \mathbb{E}\left[f\big(W^*_{\mathcal{T}^*_k},\mathcal{T}^*_k\big) 
\mathbbm{1}_{\Gamma^*_k} \mathbbm{1}_{\left\{\mathcal{T}^*_k < \infty \right\}} \right] \\
&= \mathbb{E}\left[f(W^*_{\tau^*},\tau^*) \mathbbm{1}_{\Gamma^*} 
\mathbbm{1}_{\{ K^* = k \}} \right],
\end{aligned}
\end{equation}
where for the second and the third equality we use Lemma~\ref{lemma:distreg} and the 
strong Markov property, and for the convergence we use Proposition~\ref{contprop1} and 
Corollary~\ref{contcor1}. Therefore
\begin{equation}
\label{pconteq2}
\begin{aligned}
&\left| \mathbb{E}\left[f(W^n_{\tau^n},\tau^n) \mathbbm{1}_{\Gamma^n} \right] 
- \mathbb{E}\left[f(W^*_{\tau^*},\tau^*) \mathbbm{1}_{\Gamma^*} \right] \right| \\ 
&\qquad \le
\| f \|_{\infty}\big\{ \mathbbm{P}(K^n > M) + \mathbbm{P}(K^* > M) \big\}\\ 
&\qquad\qquad
+ \sum_{k=1}^{M} \left| \mathbb{E}\left[f(W^n_{\tau^n},\tau^n) 
\mathbbm{1}_{\Gamma^n} \mathbbm{1}_{\{K^n = k\}} \right] 
- \mathbb{E}\left[f(W^*_{\tau^*},\tau^*) \mathbbm{1}_{\Gamma^*} 
\mathbbm{1}_{\{K^* = k\}} \right] \right|,
\end{aligned}
\end{equation}
and we conclude by taking $n \to \infty$, using Corollary~\ref{contcor1} 
and \eqref{pconteq1}, and taking $M \to \infty$.
\end{proof}


\end{document}